\documentclass[reqno]{amsart}
\usepackage{graphicx} 
\usepackage[noadjust]{cite}
\usepackage{amsmath}
\usepackage{amssymb}
\usepackage{amsthm}
\usepackage{xcolor}
\usepackage{mathabx}
\usepackage{tikz}
\usepackage{quiver}
\usepackage{hyperref}
\newtheorem{thm}{Theorem}[section]
\newtheorem{lemma}[thm]{Lemma}
\newtheorem{prop}[thm]{Proposition}
\newtheorem{cor}[thm]{Corollary}
\newtheorem{claim}{Claim}[thm]
\theoremstyle{definition}
\newtheorem{definition}[thm]{Definition}
\newtheorem{notation}[thm]{Notation}
\newenvironment{claimproof}[1]{\par\noindent\emph{Proof of Claim.}\space#1}{\hfill $\blacksquare$}

\title{The Morse Local-to-Global Property for Graph Products}
\author{Joshua Perlmutter}
\date{\today}

\begin{document}

\begin{abstract}
The Morse local-to-global property generalizes the local-to-global property for quasi-geodesics in a hyperbolic space. We show that graph products of infinite Morse local-to-global groups have the Morse local-to-global property. To achieve this, we generalize the maximization procedure from \cite{abbott2021largest} for relatively hierarchically hyperbolic groups with clean containers. Under mild conditions satisfied by graph products, we show that stable embeddings into a relatively hierarchically hyperbolic space are exactly those which are quasi-isometrically embedded in the top level hyperbolic space by the orbit map. This shows that graph products of any infinite groups with no isolated vertices are Morse detectable. 
\end{abstract}

\maketitle

\section{Introduction}

One fundamental property of hyperbolic space is that paths that are locally quasi-geodesic must themselves be globally quasi-geodesic. In fact, \cite[Proposition 7.2.E]{gromov1987hyperbolic} shows that this quasi-geodesic property can be taken to be the definition of a hyperbolic space. Another key property of quasi-geodesics in hyperbolic space, known as the \emph{Morse Lemma} \cite[Proposition 7.2.A]{gromov1987hyperbolic}, states that quasi-geodesics with the same endpoints must fellow-travel. Non-hyperbolic spaces can have Morse quasi-geodesics, which are quasi-geodesics that satisfy the Morse lemma. In fact, if all geodesics are uniformly Morse, then the space is hyperbolic. Given the close relationship between these two properties in a hyperbolic space, it is natural to ask if Morse quasi-geodesics in a non-hyperbolic space still satisfy the local-to-global quasi-geodesic property. The \emph{Morse local-to-global property} for metric spaces was introduced by Russell, Spriano, and Tran \cite{russell2022local} to study exactly this question.

A metric space has the Morse local-to-global property if, roughly speaking, every path which is locally a Morse quasi-geodesic must be a global Morse quasi-geodesic. A group is then Morse local-to-global if its Cayley graph has this property, as this property is a quasi-isometry invariant \cite{russell2022local}. There are many known examples of Morse local-to-global groups and spaces: direct products of infinite groups \cite{russell2022local}; hierarchically hyperbolic groups \cite{russell2022local}; groups hyperbolic relative to Morse local-to-global groups \cite{russell2022local}; injective metric spaces \cite{sisto2024morse}; and any $C'(1/9)$-small cancellation group with a $\sigma$-compact Morse boundary \cite{he2024sigma}.

There are several properties of hyperbolic spaces whose proof relies solely on the fact that local quasi-geodesics are global quasi-geodesics, and so analogs of these properties hold in a Morse local-to-global space. By using this technique and beyond, Morse local-to-global spaces have been shown to have many robust properties:
\begin{itemize}
    \item A combination theorem for stable subgroups \cite[Theorem 3.1]{russell2022local};
    \item A trichotomy law \cite[Corollary 4.9]{russell2022local};
    \item A Cartan-Hadamard-style Theorem \cite[Theorem 3.15]{russell2022local};
    \item A growth rate gap for stable subgroups \cite[Theorem 5.1]{cordes2022regularity}; 
    \item A regular language of Morse geodesic words \cite[Theorem 3.2]{cordes2022regularity};
    \item A strongly $\sigma$-compact Morse boundary \cite[Theorem 4.3]{he2024sigma}.
\end{itemize}

Morse local-to-global groups and spaces continue to be an active area of study, both to search for new properties implied by the Morse local-to-global condition, and to find new examples of Morse local-to-global groups and spaces. Russell-Spriano-Tran asked if graph products of Morse local-to-global groups are themselves Morse local-to-global \cite[Question 3]{russell2022local}. Graph products were introduced by Green \cite{green1990graph}. Generalizing the notion of a right-angled Artin group, graph products assign finitely generated groups to every vertex in a finite simplicial graph. The graph product is then the free product of all vertex groups, with the relation that elements in two different vertex groups commute if and only if their respective vertices are connected by an edge. Graph products are thus a bridge between free products and direct products of groups. Our main result answers \cite[Question 3]{russell2022local} for the case of infinite Morse local-to-global groups:

\begin{thm}
\label{intro graph prod}
Graph products of infinite Morse local-to-global groups are Morse local-to-global.
\end{thm}

In fact, we prove in Corollary \ref{multivertex} that if every vertex group in a graph product with no isolated vertices is infinite, then the group is Morse local-to-global. In particular, as long as a graph product has no isolated vertices, every vertex group could be a non-Morse local-to-global group, yet the resulting graph product will be Morse local-to-global.

Despite graph products being a rich area of study themselves, to prove Theorem \ref{intro graph prod} we instead rely on the fact that graph products are \emph{relatively hierarchically hyperbolic groups} \cite[Theorem 4.22]{berlyne2022hierarchical}. Relatively hierarchically hyperbolic groups and spaces were introduced by \cite{Behrstock_2017} and \cite{behrstock2019hierarchically} to generalize the hierarchy structure of mapping class groups. Relatively hierarchically hyperbolic groups and spaces consist of a collection of hyperbolic spaces along with three relations between them: nesting, orthogonality, and transversality. Nesting forms a partial order for which there exists a largest element, called the \emph{top level space}. The proof structure of Theorem \ref{intro graph prod} generalizes the argument of \cite[Theorem 4.20]{russell2022local}, which proves that hierarchically hyperbolic spaces are Morse local-to-global, which in turn relies on a procedure for hierarchically hyperbolic spaces introduced by Abbott-Behrstock-Durham known as \emph{maximization}. Because graph products are relatively hierarchically hyperbolic groups, we first generalize \cite[Theorem 3.7]{abbott2021largest}.

\begin{thm}
\label{intro max}
Every relatively hierarchically hyperbolic space with the bounded domain dichotomy and clean containers admits a relatively hierarchically hyperbolic structure with relatively unbounded products.
\end{thm}

The technical assumptions of bounded domain dichotomy and clean containers are mild assumptions satisfied, for example, every graph product; see Section \ref{section abd} for the precise definitions.

The new relatively hierarchically hyperbolic structure produced by Theorem \ref{intro max} is called the \emph{maximized structure}. From this point, we go on to show that the top level space in the maximized structure is a hyperbolic space which captures the geometry of the Morse quasi-geodesics in the overall space. Such a space is called a \emph{Morse detectability space} \cite{russell2022local}, and having a Morse detectability space is sufficient to prove that a space is Morse local-to-global \cite[Theorem 4.18]{russell2022local}. In our arguments we focus on \emph{stable embeddings}, of which Morse quasi-geodesics are an example. A stable embedding is a quasi-isometric embedding with the additional property that any two quasi-geodesics with endpoints in the image of the embedding are contained in the a uniform neighborhood of each other.

\begin{thm}
\label{intro stable}
Let $\mathcal{X}$ be a geodesic metric space, and let $(\mathcal{X},\mathfrak{S})$ be a relative HHS with $|\mathfrak{S}|>1$, clean containers, the bounded domain dichotomy, and unbounded minimal products. A quasi-isometric embedding $\gamma:\mathcal{Y}\to\mathcal{X}$ is a stable embedding if and only if $\pi_S\circ\gamma$ is a quasi-isometric embedding into $\mathcal{T}_S$, where $\mathcal{T}_S$ is the top level space of the maximized structure.
\end{thm}

We note that Balasubramanya-Chesser-Kerr-Mangahas-Trin simultaneously and independently prove a result similar to Theorem \ref{intro stable} that provides space which detects stable subgroups of a graph product of infinite groups with no isolated vertices \cite[Theorem 1.2]{balasubramanya2025stable}. Because their result relies on tools specific to stable subgroups rather than general stable embeddings, \cite[Theorem 1.2]{balasubramanya2025stable} alone does not imply that the space they construct is a Morse detectability space.

\subsection*{Outline}

Section \ref{sec Background} provides necessary background information for the paper. We begin by discussing quasi-geodesics and stable embeddings, cumulating in the Morse local-to-global property at the end of Section \ref{subsec MLTG}. In Section \ref{subsec HHS} we give the definition and basic properties of relatively hierarchically hyperbolic groups and spaces. We provide a more in-depth discussion on hierarchy paths in Section \ref{subsec hpaths}, including modifying the standard definition and explaining why such a modification does not preclude us from using certain results based on the original definition. In Section \ref{subsec graph prod} we provide the definition and basic properties of graph products, including the relative hierarchy structure on graph products from \cite[Theorem 4.22]{berlyne2022hierarchical}. The purpose of Section \ref{big section construction} is to generalize the methods and results of \cite{abbott2021largest} for relative HHS. To do this, in Section \ref{subsec active}, we add details to the proof of \cite[Proposition 20.1]{casals2022real} to show that it generalizes to relative HHS. From there, we generalize the maximization procedure from \cite[Theorem 3.7]{abbott2021largest} in Section \ref{section abd}, proving Theorem \ref{intro max}. Section \ref{subsec bounded proj} then generalizes \cite[Theorem 4.4]{abbott2021largest}. The main goal of Section \ref{gp} is to generalize \cite[Corollary 6.2]{abbott2021largest} for a relative HHS with certain nice properties, proving Theorem \ref{intro stable}. The final section, Section \ref{sec MLTG}, then shows that graph products of infinite groups with no isolated vertices satisfy such nice properties, and are thus Morse local-to-global, which is then used to prove Theorem \ref{intro graph prod}.

\subsection*{Acknowledgments}

The author is grateful to his PhD advisor, Carolyn Abbott. The author also wishes to thank Jacob Russell for suggesting this problem, as well as providing guidance and insight. The author was partially supported by NSF grant DMS-2106906.

\section{Background}
\label{sec Background}

\subsection{The Morse Local-to-Global Property}
\label{subsec MLTG}

We begin this section by recalling some basic definitions about quasi-isometries and metric spaces.

\begin{definition}
Let $\mathcal{X},\mathcal{Y}$ be metric spaces. A map $\gamma\colon\mathcal{Y}\to\mathcal{X}$ is a \emph{$(\lambda,\varepsilon)$-quasi-isometric embedding} if for any $x,y\in\mathcal{Y}$
$$ \frac{1}{\lambda}\cdot d_{\mathcal{Y}}(x,y)-\varepsilon\leq d_{\mathcal{X}}(\gamma(x),\gamma(y))\leq \lambda\cdot d_{\mathcal{Y}}(x,y)+\varepsilon.$$
\end{definition}

\begin{definition}
A \emph{$(\lambda,\varepsilon)$-quasi-geodesic} is a $(\lambda,\varepsilon)$-quasi-isometric embedding of a closed subset $I\subseteq\mathbb{R}$ into a metric space. A $(1,0)$-quasi-geodesic is called a \emph{geodesic}.
\end{definition}

\begin{definition}
A metric space $\mathcal{X}$ is a \emph{$(\lambda,\varepsilon)$-quasi-geodesic space} if any two points in $\mathcal{X}$ can be connected by a $(\lambda,\varepsilon)$-quasi-geodesic. If the constants $\lambda$ and $\varepsilon$ are not important, we simply call $\mathcal{X}$ a \emph{quasi-geodesic space}. Similarly, $\mathcal{X}$ is a \emph{geodesic space} if any two points in $\mathcal{X}$ can be connected by a geodesic.
\end{definition}

We will follow the convention set by \cite{russell2022local} regarding the definition of Morse. The following definition appears stronger than the standard definition attributed to \cite{gromov1987hyperbolic}, but is in fact equivalent by \cite[Lemma 2.4]{russell2022local}. 

\begin{definition}
\label{Morsedef}
Let $M\colon [1,\infty)\times[0,\infty)\to[0,\infty)$ be a function. The $(\lambda,\varepsilon)$-quasi-geodesic $\gamma\colon I\to X$ is an \emph{$(M;\lambda,\varepsilon)$-Morse quasi-geodesic} if for all $s<t$ in $I$, if $\alpha$ is a $(k,c)$-quasi-geodesic with endpoints $\gamma(s)$ and $\gamma(t)$, then the Hausdorff distance between $\alpha$ and $\gamma|_{[s,t]}$ is bounded by $M(k,c)$. The function $M$ is referred to as a \emph{Morse gauge}.
\end{definition}

The notion of Morse comes from the fact that in a hyperbolic space, all quasi-geodesics are Morse. 

\begin{prop}[{\cite[Proposition 7.2.A]{gromov1987hyperbolic}}]
\label{Gromov Morse}
Let $X$ be a $\delta$-hyperbolic space. A $(\lambda,\varepsilon)$-quasi-geodesic in $X$ is $M$-Morse, where $M$ depends on $\lambda$, $\varepsilon$, and $\delta$.
\end{prop}

Closely tied with the definition of a Morse quasi-geodesic is the notion of a stable subgroup, which was defined by \cite{durham2015convex} to generalize quasi-convex subgroups of hyperbolic groups.

\begin{definition}
Let $\mathcal{X},\mathcal{Y}$ be metric spaces and let there exist a map $\iota\colon\mathcal{Y}\to \mathcal{X}$. We say $\iota$ is an \emph{$(M;\lambda,\varepsilon)$-stable embedding} if $\iota$ is a $(\lambda,\varepsilon)$-quasi-isometric embedding and there exists a function $M\colon[1,\infty)\times[0,\infty)\to[0,\infty)$ such that any two $(k,c)$-quasi-geodesics in $\mathcal{X}$ with endpoints in $\iota(\mathcal{Y})$ are contained in the $M(k,c)$-neighborhood of each other. We call $M$ the \emph{stability gauge}.
\end{definition}

\begin{definition}
Let $G$ be a finitely generated group. A subgroup $H\leq G$ is a \emph{stable subgroup} if the inclusion $H\hookrightarrow G$ is a stable embedding for some (any) Cayley graphs of $H$ and $G$, respectively.
\end{definition}

The following lemmas are useful tools regarding the geometry of stable embeddings.

\begin{lemma}
\label{morse is stable}
Let $\mathcal{X}$ be a metric space. Any $(M;\lambda,\varepsilon)$-Morse quasi-geodesic $\gamma:I\to\mathcal{X}$ is an $(M;\lambda,\varepsilon)$-stable embedding.
\end{lemma}

\begin{proof}
Follows immediately from Definition \ref{Morsedef}.
\end{proof}

\begin{lemma}
\label{qi hyp stable}
Let $\mathcal{X}$ be a geodesic hyperbolic metric space. Any $(\lambda,\varepsilon)$-quasi-isometric embedding $\gamma:\mathcal{Y}\to\mathcal{X}$ is an $(M;\lambda,\varepsilon)$-stable embedding, for some stability gauge $M$ depending only on the hyperbolicity constant.
\end{lemma}

\begin{proof}
By definition, $\gamma$ is a $(\lambda,\varepsilon)$-quasi-isometric embedding. Fix two $(k,c)$-quasi-geodesics $\alpha$ and $\beta$ with endpoints in $\gamma(\mathcal{Y})$. Fix a geodesic $\eta$ between then endpoints of $\alpha$ and $\beta$. By Proposition \ref{Gromov Morse}, $\eta$ is $M$-Morse, for some Morse gauge depending only on the hyperbolicity constant of $\mathcal{X}$. Then by the definition of Morse, the Hausdorff distance between $\alpha$ and $\beta$ is at most $2M(k,c)$, so $\gamma$ is a $(2M;\lambda,\varepsilon)$-stable embedding.
\end{proof}

\begin{lemma}
\label{stable is morse}
Let $\mathcal{Y}$ be a geodesic metric space and let $\gamma:\mathcal{Y}\to\mathcal{X}$ be a $(M;\lambda,\varepsilon)$-stable embedding. Then for any $(k,c)$-quasi-geodesic $\alpha$ with endpoints in $\gamma(\mathcal{Y})$, $\alpha$ is contained in the $M'(k,c)$ neighborhood of a $(M';\lambda,\varepsilon)$-Morse quasi-geodesic $\beta\subset\gamma(\mathcal{Y})$, where $M'$ depends on $M$, $\lambda$, and $\varepsilon$.
\end{lemma}

\begin{proof}
Fix a $(k,c)$-quasi-geodesic $\alpha$ with endpoints $\gamma(x)$ and $\gamma(y)$ for some $x,y\in\mathcal{Y}$. The space $\mathcal{Y}$ is geodesic, so let $\eta:I\to\mathcal{Y}$ be a geodesic with endpoints $x$ and $y$. The map $\gamma$ is a $(M;\lambda,\varepsilon)$-stable embedding and therefore a quasi-isometric embedding, so $\gamma\circ\eta$ is a $(\lambda,\varepsilon)$-quasi-geodesic in $\gamma(\mathcal{Y})$. Moreover, $\gamma\circ\eta$ has endpoints $\gamma(x)$ and $\gamma(y)$. Thus $\gamma\circ\eta$ and $\alpha$ are contained in the $M(\max{(\lambda,k)},\max{(\varepsilon,c)})$-neighborhood of each other. Therefore taking $\beta=\gamma\circ\eta$ and $M'(k,c)=M(\max{(\lambda,k)},\max{(\varepsilon,c)})$ completes the proof.
\end{proof}

We now recall the definition of local quasi-geodesics and then discuss situations where such embeddings have global properties.

\begin{definition}
Let $\mathcal{X}$ be a metric space, $I\subseteq \mathbb{R}$ closed, $\lambda\geq 1$, $\varepsilon\geq 0$, and $L\geq 0$. The map $\gamma:I\to\mathcal{X}$ is an \emph{$(L;\lambda,\varepsilon)$-local-quasi-geodesic} if for any $s,t\in I$ such that $|s-t|\leq L$, the restriction $\gamma|_{[s,t]}$ is a $(\lambda,\varepsilon)$-quasi-geodesic. Moreover, if there exists a Morse gauge $M$ such that $\gamma|_{[s,t]}$ is a $(M;\lambda,\varepsilon)$-Morse quasi-geodesic, then $\gamma$ is an \emph{$(L;M;\lambda,\varepsilon)$-local Morse quasi-geodesic}.
\end{definition}

One of the most fundamental facts about hyperbolic spaces is that local-quasi-geodesics must also be global quasi-geodesics.

\begin{prop}[{\cite[Proposition 7.2.E]{gromov1987hyperbolic}}]
Let $X$ be a geodesic metric space. Then $\mathcal{X}$ is $\delta$-hyperbolic if and only if for any $\lambda\geq 1$ and $\varepsilon\geq0$ there exists $\lambda'\geq 1$, $\varepsilon'\geq 0$, and $L\geq 0$ such that any $(L;\lambda,\varepsilon)$-local quasi-geodesic is a global $(\lambda',\varepsilon')$-quasi-geodesic.
\end{prop}

It is from this fact about hyperbolic spaces that \cite{russell2022local} defined the Morse local-to-global property.

\begin{definition}
Let $\mathcal{X}$ be a quasi-geodesic metric space. We say that $\mathcal{X}$ is $\emph{Morse local-to-global}$ if for any $\lambda\geq 1$, $\varepsilon\geq 0$, and Morse gauge $M$, there exists $\lambda'\geq 1$, $\varepsilon'\geq 0$, $L\geq 0$, and a Morse gauge $M'$ such that any $(L;M;\lambda,\varepsilon)$-local Morse quasi-geodesic is a global $(M';\lambda',\varepsilon')$-Morse quasi-geodesic. A finitely generated group whose Cayley graph is Morse local-to-global is a \emph{Morse local-to-global group}.
\end{definition}

It is worth noting that a group is Morse local-to-global regardless of the choice of generating set because being Morse local-to-global is a quasi-isometry invariant, as explained in \cite{russell2022local}. The following definition illustrates an important class of Morse local-to-global groups.

\begin{definition}
A metric space $\mathcal{X}$ is \emph{Morse limited} if for every Morse gauge $M$ and $\lambda\geq 1$, $\varepsilon\geq 0$, there exists $B\geq 0$ so that every $(M;\lambda,\varepsilon)$–Morse quasi-geodesic $\gamma\colon I\to \mathcal{X}$ has $\text{diam}_{\mathcal{X}}(\gamma)\leq B$. A finitely generated group $G$ is \emph{Morse limited} if $\text{Cay}(G,S)$ is Morse limited for some finite generating set $S$.
\end{definition}

Morse limited groups are trivially Morse local-to-global and play an important role in the main results of this paper. Direct products of infinite spaces are Morse limited, so ``long" Morse quasi-geodesics in a metric space must avoid subspaces which decompose as direct products of infinite spaces. 

Another way to construct a Morse local-to-global group is by taking the free product of Morse local-to-global groups. This follows immediately from \cite[Theorem 5.1]{russell2022local}, which states that a metric space which is hyperbolic to Morse local-to-global spaces is itself a Morse local-to-global space. The fact that being Morse local-to-global is preserved under free products will be useful later.

A key technique for showing that a metric space is Morse local-to-global is to show that there exists a hyperbolic space which captures all of the Morse quasi-geodesics in the original space, as seen in the following definition.

\begin{definition}
\label{morse detectable def}
A metric space $\mathcal{X}$ is \emph{Morse detectable} if there exists a $\delta$-hyperbolic space $\mathcal{Y}$, called the \emph{Morse detectability space}, and a coarsely Lipschitz map $\pi\colon \mathcal{X}\to \mathcal{Y}$ such that for every $(\lambda,\varepsilon)$-quasi-geodesic $\gamma\colon I\to\mathcal{X}$, the following holds:
\begin{enumerate}
    \item If $\gamma$ is $M$-Morse, then $\pi\circ\gamma$ is a $(k,c)$-quasi-geodesic in $\mathcal{Y}$, where $(k,c)$ is determined by $\lambda$, $\varepsilon$, $\delta$, and $M$. 
    \item If $\pi\circ\gamma$ is a $(k,c)$-quasi-geodesic in $\mathcal{Y}$, then $\gamma$ is $M$-Morse, where $M$ is determined by $k$, $c$, $\lambda$, $\varepsilon$, and $\delta$.
\end{enumerate}
\end{definition}

Finding Morse detectability spaces was the method by which \cite{russell2022local} showed that hierarchically hyperbolic spaces were Morse local-to-global. Such a result hinges on the following fact.

\begin{thm}[{\cite[Theorem 4.18]{russell2022local}}]
\label{Mdetectable is MLTG}
If $\mathcal{X}$ is Morse detectable, then $\mathcal{X}$ is Morse local-to-global.
\end{thm}

For our main result, we will also follow this technique and ultimately show that the Cayley graph of a graph product of infinite groups is Morse detectable and hence Morse local-to-global.

\subsection{Relatively Hierarchically Hyperbolic Groups and Spaces}
\label{subsec HHS}

We now provide the definitions and basic properties of relatively hierarchically hyperbolic groups, first introduced in \cite{Behrstock_2017} and \cite{behrstock2019hierarchically}.

\begin{definition}
Let $E>0$. A quasi-geodesic metric space $\mathcal{X}$ is an \emph{$E$-relatively hierarchically hyperbolic space} ($E$-relative HHS) if there exists an index set $\mathfrak{S}$ and geodesic spaces $\{(CW,d_W)\;|\;W\in\mathfrak{S}$\} such that the following twelve axioms are satisfied. The elements of $\mathfrak{S}$ are \emph{domains} and $E$ is the \emph{hierarchy constant}. An index set and associated geodesic spaces that satisfy the following axioms are referred to as a \emph{relative HHS structure on $\mathcal{X}$}.\\

\noindent
(1) (\textbf{Projections}) For each $W\in\mathfrak{S}$, there exists a \emph{projection} $\pi_W\colon\mathcal{X}\to 2^{CW}-\emptyset$ such that for all $x\in\mathcal{X}$, the diameter of $\pi_W(x)$ in $CW$ is at most $E$. Moreover, each $\pi_W$ is $(E,E)$-coarsely Lipschitz and $CW\subseteq \mathcal{N}_E(\pi_W(\mathcal{X}))$ for all $W\in\mathfrak{S}$.\\

\noindent
(2) (\textbf{Nesting}) If $\mathfrak{S}\ne \emptyset$, then $\mathfrak{S}$ is equipped with a partial order $\sqsubseteq$ and contains a unique $\sqsubseteq$-maximal element. The geodesic space associated to this $\sqsubseteq$-maximal element is the \emph{top level space}. When $V\sqsubseteq W$, the domain $V$ is \emph{nested} in $W$. For each $W\in\mathfrak{S}$, we denote by $\mathfrak{S}_W$ the set of all $V\in\mathfrak{S}$ with $V\sqsubseteq W$. Moreover, for all $V,W\in\mathfrak{S}$ with $V \sqsubsetneq W$ there is a specified non-empty subset $\rho^V_W\subseteq CW$ with diam$(\rho^V_W)\leq E$. Call the set $\rho^V_W$ the \emph{relative projection} from $V$ to $W$.\\

\noindent
(3) (\textbf{Orthogonality}) $\mathfrak{S}$ has a symmetric relation called \emph{orthogonality}. If $V$ and $W$ are orthogonal, we write $V\perp W$ and require that $V$ and $W$ are not $\sqsubseteq$-comparable. Further, whenever $V\sqsubseteq W$ and $W\perp U$, we require that $V\perp U$. We denote by $\mathfrak{S}^\perp_W$ the set of all $V\in\mathfrak{S}$ with $V\perp W$.\\

\noindent
(4) (\textbf{Transversality}) If $V,W\in\mathfrak{S}$ are not orthogonal and neither is nested in the other, then we write $V\pitchfork W$ and say the domains $V,W$ are \emph{transverse}. Additionally, for all $V,W\in\mathfrak{S}$ with $V\pitchfork W$ there are non-empty sets $\rho^V_W\subseteq CW$ and $\rho^W_V\subseteq CV$ each of diameter at most $E$. Similarly to the projection axiom, call the set $\rho^V_W$ the \emph{relative projection} from $V$ to $W$.\\

\noindent
(5) (\textbf{Hyperbolicity}) For each $W\in\mathfrak{S}$, if $CW$ is not $E$-hyperbolic, then $W$ is $\sqsubseteq$-minimal.\\

\noindent
(6) (\textbf{Finite Complexity}) The cardinality of any set of pairwise $\sqsubseteq$-comparable elements is at most $E$.\\

\noindent
(7) (\textbf{Containers}) For each $W\in\mathfrak{S}$ and $U\in\mathfrak{S}_W$ with $\mathfrak{S}_W\cap \mathfrak{S}^\perp_U\ne\emptyset$, there exists $Q\sqsubsetneq W$ such that $V\sqsubseteq Q$ whenever $V\in\mathfrak{S}_W\cap \mathfrak{S}^\perp_U$. The domain $Q$ is referred to as a \emph{container of $U$ in $W$}.\\

\noindent
(8) (\textbf{Uniqueness}) There exists a function $\theta\colon[0,\infty)\to[0,\infty)$ so that for all $r\geq 0$, if $x,y\in\mathcal{X}$ and $\text{d}_{\mathcal{X}}(x,y)\geq \theta(r)$, then there exists $W\in\mathfrak{S}$ such that $\text{d}_W(\pi_W(x),\pi_W(y))\geq r$. We call $\theta$ the \emph{uniqueness function of $\mathfrak{S}$}.\\

\noindent
(9) (\textbf{Bounded Geodesic Image}) For all $x,y\in\mathcal{X}$ and $V,W\in\mathfrak{S}$ with $V\sqsubsetneq W$, if $\text{d}_V(\pi_V(x),\pi_V(y))\geq E$, then for any $CW$-geodesic $[\pi_W(x),\pi_W(y)]$, the intersection
$$ [\pi_W(x),\pi_W(y)]\cap \mathcal{N}_E(\rho^V_W)\ne\emptyset.$$

\noindent
(10) (\textbf{Large Links}) For all $W\in\mathfrak{S}$ and $x,y\in\mathcal{X}$, there exists $\{V_1,\dots,V_m\}\subseteq \mathfrak{S}_W\setminus\{W\}$ such that $m\leq E\cdot \text{d}_W(\pi_W(x),\pi_W(y))+E$, and for all $U\in\mathfrak{S}_W\setminus\{W\}$, either $U\in\mathfrak{S}_{V_i}$ for some $i$, or $\text{d}_U(\pi_U(x),\pi_U(y))\leq E$.\\

\noindent
(11) (\textbf{Consistency}) If $V\pitchfork W$, then
$$ \min\{\text{d}_W(\pi_W(x),\rho^V_W),\text{d}_V(\pi_V(x),\rho^W_V)\}\leq E $$
for all $x\in\mathcal{X}$. Further, if $U\sqsubseteq V$ and either $V \sqsubsetneq W$ or $V\pitchfork W$ and $W\perp U$, then $\text{d}_W(\rho^U_W,\rho^V_W)\leq E$.\\

\noindent
(12) (\textbf{Partial Realization}) If $\{V_i\}$ is a finite collection of pairwise orthogonal elements of $\mathfrak{S}$ and $p_i\in CV_i$ for each $i$, then there exists $x\in\mathcal{X}$ such that:

\noindent
$\bullet$ $\text{d}_{V_i}(\pi_{V_i}(x),p_i)\leq E$ for all $i$; 

\noindent
$\bullet$ for each $i$ and each $W\in\mathfrak{S}$, if $V_i\sqsubsetneq W$ or $W\pitchfork V_i$, then $\text{d}_W(\pi_W(x),\rho^{V_i}_W)\leq E$.
\end{definition}

The notion of relatively hierarchically hyperbolic spaces can also be applied to finitely generated groups.

\begin{definition}
Let $G$ be a finitely generated group and let $\mathcal{X}$ be the Cayley graph of $G$ with respect to some finite generating set. The group $G$ is an \emph{$E$-relatively hierarchically hyperbolic group} ($E$-relative HHG) if:\\

\noindent
(1) There exists an index set $\mathfrak{S}$ such that the pair $(\mathcal{X},\mathfrak{S})$ is a relative HHS.\\

\noindent
(2) There is an action of $G$ on $\mathfrak{S}$ by bijections such that the relations $\sqsubseteq$, $\perp$, and $\pitchfork$ are preserved and $\mathfrak{S}$ contains finitely many $G$-orbits.\\

\noindent
(3) For all $W\in\mathfrak{S}$ and $g\in G$, there exists an isometry $g_W\colon CW\to C(gW)$ such that:\\

\noindent
$\bullet$ For all $h\in G$, the map $(gh)_W=g_{hW}\circ h_W$.

\noindent
$\bullet$ For all $x\in\mathcal{X}$, 
$ d_{gW}(g_W\circ\pi_W(x),\pi_{gW}(g\cdot x))\leq E$.

\noindent
$\bullet$ For all $V\in\mathfrak{S}$, if $V\pitchfork W$ or $V\sqsubsetneq W$, then $d_{gW}(g_W(\rho^V_W),\rho^{gV}_{gW})\leq E$.
\end{definition}

One of the most fundamental properties of a relative HHS is that distances in the underlying space can be determined based on projections to the associated geodesic spaces. This concept is formalized in the following theorem, known as the distance formula.

\begin{thm}[{\cite[Theorem 6.10]{behrstock2019hierarchically}}]
\label{distance formula} 
Let $(\mathcal{X},\mathfrak{S})$ be an $E$-relative HHS. Then there exists a constant $s_0\geq 0$, depending on $E$ and known as the distance formula threshold, such that for any $s\geq s_0$, there exist constants $K,C$ (depending on $s$ and $E$) such that for any $x,y\in\mathcal{X}$,
$$ d_{\mathcal{X}}(x,y)\asymp_{K,C} \sum_{U\in\mathfrak{S}}\{\!\{ d_U(\pi_U(x),\pi_U(y)) \}\!\}_s.$$
\end{thm}

Because the distance formula has a minimum threshold for which distance in an associated geodesic space affects the distance in the underlying space, it is natural to create a definition which includes only those domains whose associated geodesic spaces have large projections for a given pair of points in $\mathcal{X}$.

\begin{definition}
Let $(\mathcal{X},\mathfrak{S})$ be a relative HHS. A domain $U\in\mathfrak{S}$ is \emph{$C$-relevant} for the points $x,y\in\mathcal{X}$ if 
$$d_U(\pi_U(x),\pi_U(y))\geq C.$$
\end{definition}

Another important concept in the world of relatively hierarchical hyperbolic spaces is that of hierarchical quasi-convexity. This notion generalizes the idea of quasi-convex subspaces of hyperbolic spaces.

\begin{definition}
\label{hqconvex}
Let $(\mathcal{X},\mathfrak{S})$ be an $E$-relative HHS. Then $\mathcal{Y}\subseteq\mathcal{X}$ is \emph{$k$-hierarchically quasi-convex}, for some $k\colon[0,\infty)\to[0,\infty)$, if the following hold:
\begin{enumerate}
    \item for all $U\in\mathfrak{S}$ with $CU$ an $E$-hyperbolic space, the projection $\pi_U(\mathcal{Y})$ is $k(0)$-quasi-convex;
    \item for each $\sqsubseteq$-minimal $U\in\mathfrak{S}$ for which $CU$ is not an $E$-hyperbolic space, either $CU=\mathcal{N}_{k(0)}(\pi_U(\mathcal{Y}))$ or $\text{diam}(\pi_U(\mathcal{Y}))\leq k(0)$; and
    \item for every $x\in\mathcal{X}$ and every $R\geq 0$, if $d_U(\pi_U(x),\pi_U(\mathcal{Y}))\leq R$ for every $U\in\mathfrak{S}$, then $d_{\mathcal{X}}(x,\mathcal{Y})\leq k(R)$.
\end{enumerate}
\end{definition}

Closely related to the concept of hierarchically quasi-convex spaces in a relative HHS are subspaces known as the nested and orthogonal partial tuples. Together, these spaces form the \emph{standard product regions} in a relative HHS, which can act as barriers to hyperbolicity in the space. To define the standard product regions, we must first consider certain collections of points in the associated geodesic spaces called \emph{consistent tuples}.

\begin{definition}
Let $(\mathcal{X},\mathfrak{S})$ be an $E$-relative HHS. Let
$$\Vec{b}\in\prod_{U\in\mathfrak{S}}2^{CU}$$
be a tuple such that each coordinate $b_U$ has diameter at most $E$ in $CU$. The tuple $\Vec{b}$ is \emph{consistent} if 
\begin{itemize}
    \item for any $V,W\in\mathfrak{S}$ such that $V\pitchfork W$, $\text{min}\{d_W(b_W,\rho^V_W),\;d_V(b_V,\rho^W_V)\}\leq E;$
    \item for any $V,W\in\mathfrak{S}$ such that $V\sqsubsetneq W$, $ d_W(b_W,\rho^V_W)\leq E$.
\end{itemize}
\end{definition}

We are now ready to define the partial tuple sets which together form the standard product region.

\begin{definition}
Let $(\mathcal{X},\mathfrak{S})$ be a relative HHS. For any $U\in\mathfrak{S}$, 
\begin{itemize}
    \item the \emph{nested partial tuple} $\mathbf{F}_U$ is the set of consistent tuples in $\displaystyle{\prod_{V\in\mathfrak{S}_U} 2^{CV}}$;
    \item the \emph{orthogonal partial tuple} $\mathbf{E}_U$ is the set of consistent tuples in $\displaystyle{\prod_{V\in\mathfrak{S}_U^{\perp}} 2^{CV}}$.
\end{itemize}
\end{definition}

\begin{prop}
\label{prod map}
Let $(\mathcal{X},\mathfrak{S})$ be an $E$-relative HHS. For any $U\in\mathfrak{S}$, there exists $C\geq 0$, depending only on $E$, such that for any $\Vec{a}\in\mathbf{F}_U$ and $\Vec{b}\in\mathbf{E}_U$, there exists $x\in\mathcal{X}$ such that for any $V\in\mathfrak{S}$,
\begin{itemize}
    \item if $V\sqsubseteq U$, then $d_V(x,a_V)\leq C;$
    \item if $V\perp U$, then $d_V(x,b_V)\leq C;$
    \item if $V\pitchfork U$ or $U\sqsubsetneq V$, then $d_V(x,\rho^U_V)\leq C.$
\end{itemize}
Moreover, there exists a well-defined map $\phi_U\colon\mathbf{F}_U\times\mathbf{E}_U\to\mathcal{X}$ by setting $\phi_U(\Vec{a},\Vec{b})=x$.
\end{prop}

\begin{proof}
The argument in \cite[Construction 5.10]{behrstock2019hierarchically} goes through verbatim.
\end{proof}

\begin{definition}
Let $(\mathcal{X},\mathfrak{S})$ be a relative HHS. For any $U\in\mathfrak{S}$, let $\phi_U$ be the map from Proposition \ref{prod map}. Then $\phi_U(\mathbf{F}_U\times\mathbf{E}_U)$ is the \emph{product region for $U$}, denoted $\mathbf{P}_U$.
\end{definition}

\begin{definition}
Let $(\mathcal{X},\mathfrak{S})$ be a relative HHS. For any $U\in\mathfrak{S}$, let $\phi_U$ be the map from Proposition \ref{prod map}.  For any $\Vec{e}\in\mathbf{E}_U$ and any $\Vec{f}\in\mathbf{F}_U$, call $\phi_U(\mathbf{F}_U\times\{\Vec{e}\})$ and $\phi_U(\{\Vec{f}\}\times\mathbf{E}_U)$ \emph{slices}.
\end{definition}

\begin{notation}
We abuse notation slightly by dropping $\phi_U$ when referring to slices, denoting them as $\mathbf{F}_U\times\{\Vec{e}\}$ and $\{\Vec{f}\}\times\mathbf{E}_U$, respectively.
\end{notation}

It is worth noting that for a fixed $U$, the distance formula implies that any slices $\mathbf{F}_U\times\{\Vec{e}_1\}$ and $\mathbf{F}_U\times\{\Vec{e}_2\}$ are uniformly quasi-isometric. The same is true for slices of $\mathbf{E}_U$.

It is important to note that based on its definition, for any $U,V\in\mathfrak{S}$ with $U\perp V$, the space $CU$ coarsely embeds in $\mathbf{F}_U$ and the space $CV$ coarsely embeds in $\mathbf{E}_U$. Thus if $CU$ and $CV$ are infinitely diameter spaces, respectively, then any slices $\mathbf{F}_U\times\{\Vec{e}\}$ and $\{\Vec{f}\}\times\mathbf{E}_U$ are unbounded as subspaces of $\mathcal{X}$. In this case, we call $\mathbf{F}_U$ and $\mathbf{E}_U$ unbounded. Moreover, if both $\mathbf{F}_U$ and $\mathbf{E}_U$ are unbounded, then $\mathbf{P}_U$ decomposes as a direct product with unbounded factors. 

Because product regions can be direct products with unbounded factors, they can inhibit the hyperbolicity of $\mathcal{X}$. By coning off the product regions, we obtain a new, related space called the \emph{factored space}.

\begin{definition}
Let $(\mathcal{X},\mathfrak{S})$ be a relative HHS. Let $\mathfrak{T}\subset\mathfrak{S}$. The \emph{factored space} $\widehat{\mathcal{X}}_{\mathfrak{T}}$ is the cone-off of the slices $\mathbf{F}_V\times\{\Vec{e}\}$ in $\mathcal{X}$ for all $\Vec{e}\in\mathbf{E}_V$ and all $V\in\mathfrak{T}$.
\end{definition}

By the construction in section 1.2.1 of \cite{Behrstock_2017}, the any slices $\mathbf{F}_U\times\{\Vec{e}\}$, any slices $\{\Vec{f}\}\times\mathbf{E}_U$, and $\mathbf{P}_U$ are all uniformly hierarchically quasi-convex. Moreover, this construction proves the existence of gate maps for hierarchically quasi-convex subspaces of $\mathcal{X}$.

\begin{prop}
\label{general gates}
Let $(\mathcal{X},\mathfrak{S})$ be an $E$-relative HHS. Let $\mathcal{Y}\subset\mathcal{X}$ be $k$-hierarchically quasi-convex. Then there exists a constant $\mu\geq 1$, depending only on $E$ and $k$, such that for any $x\in\mathcal{X}$ there exists a point $y\in\mathcal{Y}$ such that for any $U\in\mathfrak{S}$
\begin{itemize}
    \item if $CU$ is $E$-hyperbolic, then $d_U(\pi_U(y),p_U\circ\pi_U(x))\leq\mu$ where $p_U$ is the coarse projection of $CU$ onto $\pi_U(\mathcal{Y})$; or
    \item if $CU$ is not $E$-hyperbolic and $\pi_U\colon\mathcal{Y}\to CU$ is $k(0)$-coarsely surjective, then $d_U(\pi_U(y),\pi_U(x))\leq\mu$.
\end{itemize}
Then there exists a map $\mathfrak{g}_{\mathcal{Y}}\colon\mathcal{X}\to\mathcal{Y}$ defined by $\mathfrak{g}_{\mathcal{Y}}(x)=y$. Moreover, the map $\mathfrak{g}_{\mathcal{Y}}$ is $(\mu,\mu)$-coarsely Lipschitz.
\end{prop}

\begin{proof}
The existence of the constant $\mu$ follows from the construction in \cite[Section 1.2.1]{Behrstock_2017}. The fact that the map is coarsely Lipschitz follows verbatim from the argument of \cite[Lemma 5.5]{behrstock2019hierarchically} (possibly enlarging the constant $\mu$).
\end{proof}

\begin{definition}
Let $(\mathcal{X},\mathfrak{S})$ be a relative HHS. For any hierarchically quasi-convex space $\mathcal{Y}\subseteq\mathcal{X}$, the map $\mathfrak{g}_{\mathcal{Y}}$ is the \emph{gate map onto $\mathcal{Y}$}.
\end{definition}

Compared to general hierarchically quasi-convex subspaces, the extra structure and terminology associated to product regions allows the properties of their gate maps to be conveyed in clearer terms. The following proposition follows the language of \cite[Proposition 2.23]{russell2022hierarchical}.

\begin{prop}[{\cite[Lemma 5.5]{behrstock2019hierarchically}},{\cite[Lemma 1.20]{behrstock2020quasiflats}}]
\label{russell gates}
Let $(\mathcal{X},\mathfrak{S})$ be a relative HHS. Then there exists a constant $\mu\geq 1$ such that for any $U\in\mathfrak{S}$, the gate map $\mathfrak{g}_U\colon\mathcal{X}\to\mathbf{P}_U$ is such that
\begin{enumerate}
    \item $\mathfrak{g}_U$ is $(\mu,\mu)$-coarsely Lipschitz;
    \item for all $p\in\mathbf{P}_U$, $d_{\mathcal{X}}(\mathfrak{g}_U(p),p)\leq \mu$;
    \item for all $x\in\mathcal{X}$ and $V\in\mathfrak{S}$, $d_V(\pi_V\circ\mathfrak{g}_U(x),\rho^U_V)\leq\mu$ if $U\perp V$ or $U\sqsubsetneq V$, and $d_V(\pi_V\circ\mathfrak{g}_U(x),\pi_V(x))\leq\mu$ otherwise; and
    \item for all $x\in\mathcal{X}$ and $p\in\mathbf{P}_U$, $d_{\mathcal{X}}(x,\mathfrak{g}_U(x))+d_{\mathcal{X}}(\mathfrak{g}_U(x),p)\leq \mu\cdot d_{\mathcal{X}}(x,p)+\mu$.
\end{enumerate}
\end{prop}

\subsection{Hierarchy Paths}
\label{subsec hpaths}

A remarkable fact about relatively hierarchically hyperbolic spaces is that any two points can be connected with a uniform quality quasi-geodesic, known as a hierarchy path. Hierarchy paths project nicely to all associated geodesic spaces in the relative HHS and are one of the most important tools for studying Morse quasi-geodesics in a relative HHS.

For our arguments, we will need a slightly stronger definition of hierarchy paths than that of \cite[Definition 4.2]{behrstock2019hierarchically}. The construction of hierarchy paths in the proof of \cite[Theorem 6.11]{behrstock2019hierarchically} satisfies this stronger condition, as shown in Proposition \ref{hpaths}.

\begin{definition}
\label{unparam qgeo}
Let $X$ be a metric space. Let $f\colon[0,\ell]\to X$ be a quasi-geodesic. Then $f$ is a \emph{$(\lambda,\lambda)$-unparametrized} \emph{quasi-geodesic} if there exists an $L\in\mathbb{N}$ and a strictly increasing function $g\colon[0,L]\to[0,\ell]$ such that
\begin{itemize}
    \item $g(0)=0$,
    \item $g(L)=\ell$,
    \item $f\circ g$ is a $(\lambda,\lambda)$-quasi-geodesic, and
    \item $\forall\;j\in[0,L-1]\cap\mathbb{Z}$, $d_X(f\circ g(j),f\circ g(j+1))\leq \lambda$.
\end{itemize}

\end{definition}

For the following definition, note that the third condition is the extra property required for hierarchy paths in this paper.

\begin{definition}
\label{hdef}
For $\lambda\geq 1$, a (not necessarily continuous) path $\gamma\colon[0,\ell]\to\mathcal{X}$ is a \emph{$(\lambda,\lambda)$-hierarchy path} if the following are satisfied:
\begin{enumerate}
\item $\gamma$ is a $(\lambda,\lambda)$-quasi-geodesic.
\item For any $U\in\mathfrak{S}$, the path $\pi_U\circ\gamma$ is an unparametrized $(\lambda,\lambda)$-quasi-geodesic.
\item For any $U\in\mathfrak{S}$, $\pi_U(\gamma)$ is contained in the $\lambda$-neighborhood of a geodesic connecting $\pi_U\circ\gamma(0)$ and $\pi_U\circ\gamma(\ell)$ in $CU$.
\end{enumerate}
\end{definition}

The following Proposition is equivalent to \cite[Remark 2.9]{tao2024property}, which was stated without proof. We provide a proof here for completeness.

\begin{prop}
\label{hpaths}
Given a relative HHS $(\mathcal{X},\mathfrak{S})$, there exists a constant $D>0$ such that for any two points $x,y\in\mathcal{X}$, there exists a $(D,D)$-hierarchy path connecting $x$ and $y$ in the sense of Definition \ref{hdef}.
\end{prop}

\begin{proof}
Fix $x,y\in\mathcal{X}$. Following the proof of \cite[Theorem 6.11]{behrstock2019hierarchically}, for a given $\theta\geq 0$ and for each $U\in\mathfrak{S}$, fix a geodesic $\gamma_U$ from $\pi_U(x)$ to $\pi_U(y)$ in $CU$. Define 
$$M_{\theta}(x,y):=\{p\in\mathcal{X}\;|\;\forall\;U\in\mathfrak{S},\;d_U(\pi_U(p),\gamma_U)\leq\theta\}.$$ 
 \cite[Proposition 6.15]{behrstock2019hierarchically} implies that $(M_{\theta}(x,y),\mathfrak{S})$ is an HHS with the relations as in $(\mathcal{X},\mathfrak{S})$ and uniform constants not depending on $x$ and $y$. For $(M_{\theta}(x,y),\mathfrak{S})$, however, the hyperbolic spaces are the geodesics $\gamma_U\subset CU$. Additionally, the projections $\pi'_U$ in $(M_{\theta}(x,y),\mathfrak{S})$ are equal to $\pi_U\circ r$, where $r\colon\mathcal{X}\to M_{\theta}(x,y)$ is a $C$-coarsely Lipschitz retraction given by \cite[Lemma 6.12]{behrstock2019hierarchically}. Thus, by applying \cite[Theorem 4.4]{behrstock2019hierarchically}, there exists a $D_0$ such that $x$ and $y$ are connected by a $(D_0,D_0)$-hierarchy path $\alpha\colon[0,\ell]\to M_{\theta}(x,y)$ in the sense of \cite[Definition 4.2]{behrstock2019hierarchically}.

We now wish to show that $\alpha$ satisfies the three conditions from Definition \ref{hdef} for $(\mathcal{X},\mathfrak{S})$. For condition (1), $\alpha$ is a $(D_0,D_0)$-hierarchy path in $(M_{\theta}(x,y),\mathfrak{S})$, so it is a $(D_0,D_0)$-quasi-geodesic in $M_{\theta}(x,y)$. Since $M_{\theta}(x,y)$ is a subspace of $\mathcal{X}$ with the subspace metric, $\alpha$ is a $(D_0,D_0)$-quasi-geodesic in $\mathcal{X}$, satisfying the first condition.

For condition (2), we want to show that for any $U\in\mathfrak{S}$, $\pi_U(\alpha)$ is a $(D',D')$-unparametrized quasi-geodesic for some $D'$. Without loss of generality, let $[0,\ell]$ be the domain of $\alpha$. Because $\alpha$ is a $(D_0,D_0)$-hierarchy path in $(M_{\theta}(x,y),\mathfrak{S})$, $\pi'_U(\alpha)$ is a $(D_0,D_0)$-unparametrized quasi-geodesic in $\gamma_U\subset CU$. Therefore, by the definition of unparametrized quasi-geodeisc, there exists a strictly increasing function $g\colon[0,L]\to[0,\ell]$ with $L\in\mathbb{N}$ such that $g(0)=0$, $g(L)=\ell$, $\pi'_U\circ \alpha\circ g$ is a $(D_0,D_0)$-quasi-geodesic in $\gamma_U\subset CU$, and for each $j\in[0,L-1]\cap\mathbb{Z}$, 
$$ d_U(\pi'_U\circ\alpha\circ g(j),\pi'_U\circ\alpha\circ g(j+1))\leq D_0.$$
We now show that for the same function $g$, the composition $\pi_U\circ\alpha\circ g$ is a $(D',D')$-quasi-geodesic in $CU$ and for each $j\in[0,L-1]\cap\mathbb{Z}$, 
$$ d_U(\pi_U\circ\alpha\circ g(j),\pi_U\circ\alpha\circ g(j+1))\leq D'.$$
Fix two points $t_1,t_2\in[0,L]$. Since $\pi'_U\circ\alpha\circ g$ is a quasi-geodesic and $\pi'_U=\pi_U\circ r$,
\begin{equation}
\label{piuqgeo}
\frac{1}{D_0}|t_2-t_1|-D_0\leq d_U(\pi_U\circ r \circ\alpha\circ g(t_1),\pi_U\circ r\circ\alpha\circ g(t_2))\leq D_0|t_2-t_1|+D_0.
\end{equation}
For one side of the inequality, using \eqref{piuqgeo} and the fact that $\pi_U$ is $(E,E)$-coarsely Lipschitz,
\begin{align*}
d_U(\pi_U\circ\alpha\circ g(t_1),\pi_U\circ\alpha\circ g(t_2))\leq &
d_U(\pi_U\circ\alpha\circ g(t_1),\pi_U\circ r\circ\alpha\circ g(t_1))\\
&+ d_U(\pi_U\circ r\circ \alpha\circ g(t_1), \pi_U\circ r\circ\alpha\circ g(t_2))\\
&+d_U(\pi_U\circ r\circ\alpha\circ g(t_2),\pi_U\circ\alpha\circ g(t_2))\\
\leq & D_0|t_2-t_1|+D_0\\
&+ E\cdot d_{\mathcal{X}}(\alpha\circ g(t_1), r\circ\alpha\circ g(t_1))+E\\
&+E\cdot d_{\mathcal{X}}(r\circ\alpha\circ g(t_2),\alpha\circ g(t_2))+E\\
\leq & D_0|t_2-t_1|+D_0 + 2(EC+E).
\end{align*}
For the final inequality above, we used the fact that $\alpha\subset M_{\theta}(x,y)$ and the definition of $r$. For the other side of the inequality,
\begin{align*}
d_U(\pi_U\circ\alpha\circ g(t_1),\pi_U\circ\alpha\circ g(t_2))\geq &
d_U(\pi_U\circ r\circ \alpha\circ g(t_1), \pi_U\circ r\circ\alpha\circ g(t_2))\\
&-d_U(\pi_U\circ\alpha\circ g(t_1),\pi_U\circ r\circ\alpha\circ g(t_1))\\
&-d_U(\pi_U\circ r\circ\alpha\circ g(t_2),\pi_U\circ\alpha\circ g(t_2))\\
\geq & \frac{1}{D_0}|t_2-t_1|-D_0\\
&-E\cdot d_{\mathcal{X}}(\alpha\circ g(t_1), r\circ\alpha\circ g(t_1))-E\\
&-E\cdot d_{\mathcal{X}}(r\circ\alpha\circ g(t_2),\alpha\circ g(t_2))-E\\
\geq & \frac{1}{D_0}|t_2-t_1|-D_0-2(EC+E).
\end{align*}
Therefore $\pi_U\circ\alpha\circ g$ is a $(D_0,D_0+2(EC+E))$-quasi-geodesic in $CU$. For the final component of condition (2), fix $j\in[0,L-1]\cap\mathbb{Z}$. Because $\pi_U\circ\alpha\circ g$ is a $(D_0,D_0+2(EC+E))$-quasi-geodesic, we have
\begin{align*}
d_U(\pi_U\circ\alpha\circ g(j), \pi_U\circ\alpha\circ g(j+1))&\leq D_0|j+1-j|+D_0+2(EC+E)\\
&= 2(D_0+EC+E).
\end{align*}
Thus this condition is satisfied for $D'=2(D_0+EC+E)$.

Finally, for condition (3), by construction $\alpha$ lies entirely in $M_{\theta}(x,y)$. Therefore, by the definition of $M_{\theta}(x,y)$, for any $U\in\mathfrak{S}$, $\pi_U(\alpha)$ is contained in the $(\theta+1)$-neighborhood of $\gamma_U$, which is a geodesic connecting $\pi_U(x)$ and $\pi_U(y)$. Thus by choosing the constant $D=\max\{\theta,D'\}$, $\alpha$ is a $(D,D)$-hierarchy path connecting $x$ and $y$ in the sense of Definition \ref{hdef}.
\end{proof}

Although not used in this paper, it is of independent interest that the construction of the hierarchy path between two points in any HHS in \cite[Theorem 4.4]{behrstock2019hierarchically} also satisfies condition (3) in Definition \ref{hdef}. We formalize this in the following proposition.

\begin{prop}
Given an HHS $(\mathcal{X},\mathfrak{S})$, there exists a constant $D>0$ such that for any two points $x,y\in\mathcal{X}$, there exists a $(D,D)$-hierarchy path connecting $x$ and $y$ in the sense of Definition \ref{hdef}.
\end{prop}

\begin{proof}
By \cite[Proposition 4.12]{behrstock2019hierarchically}, there exists a $K>0$ such that $x$ and $y$ are connected by a path $\gamma$ which is $(K,K)$-good for all $U\in\mathfrak{S}$ in the terminology of \cite{behrstock2019hierarchically}. In the proof of \cite[Proposition 4.12]{behrstock2019hierarchically}, there exists $\theta$ and $K$ such that the $(K,K)$-good path between $x$ and $y$ lies entirely in $H_{\theta}(x,y)$, which is defined to be the set of all $p\in\mathcal{X}$ such that for any $W\in\mathfrak{S},\;\pi_W(p)$ lies at distance at most $\theta$ from a geodesic in $CW$ joining $\pi_W(x)$ to $\pi_W(y)$. Given $K$, \cite[Lemma 4.18]{behrstock2019hierarchically} provides an $r>0$ such that any $K$-monotone, $(r,K)$-proper discrete path connecting $x$ and $y$ is a $(\lambda,\lambda)$-hierarchy path, for $\lambda$ a function of $r,K,$ and the HHS constants. \cite[Lemma 4.11]{behrstock2019hierarchically} modifies $\gamma$ by taking a subpath $\gamma'$ which is a $K$-monotone, $(r,K)$-proper discrete path. Thus, $\gamma'$ is a $(\lambda,\lambda)$-hierarchy path contained in $H_{\theta}(x,y)$. By taking $D$ greater than all the above constants, $\gamma'$ is a $(D,D)$-hierarchy path that satisfies condition (3) in Definition \ref{hdef}.
\end{proof}

For the remainder of this paper, we use the term ``hierarchy path" in the sense of Definition \ref{hdef}. In light of Proposition \ref{hpaths}, this causes no loss of generality.

\subsection{Graph Products}
\label{subsec graph prod}

The main result of this paper is showing that graph products of infinite Morse local-to-global groups are Morse local-to-global, and so this section presents some of the fundamental aspects of graph products. We do wish to note though, that the majority of the work done in this paper is on relatively hierarchically hyperbolic groups, of which graph products represent a specific subset. Thus the following introduction to graph products will only cover the essential tools used in this paper, while the study of graph products as a whole goes far beyond the scope of this paper.

\begin{definition}
Let $\Gamma$ be a finite simplicial graph with vertex set $V(\Gamma)$ and edges $E(\Gamma)$. To each vertex $v\in V(\Gamma)$ associate a finitely generated group $G_v$. The \emph{graph product} $G_{\Gamma}$ is defined as follows
$$ G:=\left. \left(\bigast_{v\in V(\Gamma)} G_v \right) \middle / \langle\!\langle \{ [g,h] \;|\; g\in G_v,\;h\in G_u,\; \{v,u\}\in E(\Gamma)\} \rangle\!\rangle \right. .$$
\end{definition}

Graph products are thus an intermediate construction between the free and direct product of groups. The reason they can be studied in the context of relatively hierarchically hyperbolic groups is due to the following result.

\begin{thm}[{\cite[Theorem 4.22]{berlyne2022hierarchical}}]
\label{berlyne russell}
Let $G_{\Gamma}$ be a graph product of finitely generated groups. Then $G_{\Gamma}$ is a relatively hierarchically hyperbolic group.
\end{thm}

The proof of \cite[Theorem 4.22]{berlyne2022hierarchical} is constructive, and as such we rely on that specific hierarchy structure to make further conclusions about graph products in Section \ref{sec MLTG}. We will now discuss the key components that comprise the hierarchy structure of graph products as was done in \cite{berlyne2022hierarchical}. We begin by recalling some basic definitions about graphs.

\begin{definition}
Let $\Gamma$ be a finite simplicial graph. A subgraph $\Lambda\subseteq \Gamma$ is \emph{induced} if any vertices $v,u\in\Lambda$ are connected by an edge in $\Lambda$ if they were connected by an edge in $\Gamma$.
\end{definition}

\begin{definition}
Let $\Gamma$ be a finite simplicial graph. Let $\Lambda\subseteq \Gamma$ be an induced subgraph. The \emph{link} of $\Lambda$, denoted $\text{lk}(\Lambda)$ is the induced subgraph of $\Gamma-\Lambda$ whose vertices are connected to every vertex of $\Lambda$ in $\Gamma$. The \emph{star} of $\Lambda$, denoted $\text{st}(\Lambda)$ is the induced subgraph of $\Gamma$ given by $\Lambda \cup \text{lk}(\Lambda)$.
\end{definition}

With these basic definitions in mind, we can now describe the domains, as well as the nesting and orthogonality relations in the relatively hierarchically hyperbolic structure of graph products. For the next definition, it is important to note that each induced subgraph $\Lambda\subseteq\Gamma$ induces a subgroup $G_{\Lambda}\leq G_{\Gamma}$, which is also a graph product.

\begin{notation}
Let $G_{\Gamma}$ be a graph product and let $\Lambda\subseteq \Gamma$ be an induced subgraph. For any $g\in G_{\Gamma}$, let $g\Lambda$ denote the coset $gG_{\Lambda}$.
\end{notation}

\begin{definition}
Let $G_{\Gamma}$ be a graph product and let $\Lambda\subseteq\Gamma$ be an induced subgraph. For any $g,h\in G_{\Gamma}$, the cosets $g\Lambda$ and $h\Lambda$ are \emph{parallel} if $g^{-1}h\in G_{\text{st}(\Lambda)}$. The equivalence class of parallel cosets is called a \emph{parallelism class}, and is denoted $[g\Lambda]$.
\end{definition}

\begin{thm}[{\cite[Theorem 4.22]{berlyne2022hierarchical}}]
\label{graph prod relations}
Let $G_{\Gamma}$ be a graph product. Then $G_{\Gamma}$ has a relatively hierarchically hyperbolic structure where
\begin{itemize}
    \item \textbf{Domains:} the domains are parallelism classes, so the index set is $\mathfrak{S} = \{[g\Lambda]\;|\;g\in G_{\Gamma},\; \Lambda\subseteq\Gamma\}$;
    \item \textbf{Nesting:} $[g\Lambda]\sqsubseteq [h\Omega]$ if and only if $\Lambda\subseteq\Omega$ and there exists a group element $k\in G_{\Gamma}$ such that $[g\Lambda]=[k\Lambda]$ and $[h\Omega]=[k\Omega]$; and
    \item \textbf{Orthogonality:} $[g\Lambda]\perp[h\Omega]$ if and only if $\Lambda\subseteq\text{lk}(\Omega)$ and and there exists a group element $k\in G_{\Gamma}$ such that $[g\Lambda]=[k\Lambda]$ and $[h\Omega]=[k\Omega]$.
\end{itemize}
\end{thm}

This background on graph products and their relatively hierarchically hyperbolic structure, abeit brief, is sufficient for the proofs in Section \ref{sec MLTG} of this paper to be self-contained.

\section{Constructing a Maximized Relative HHS Structure}
\label{big section construction}

The goal of this section is to generalize the construction of \cite{abbott2021largest} for relatively hierarchically hyperbolic spaces. In particular, we will show that if a relative HHS with clean containers satisfies the bounded domain dichotomy, then it admits a relative HHS structure with \emph{relatively unbounded products}, which is the analog of unbounded products from \cite{abbott2021largest} for a relative HHS; see Theorem \ref{clean containers ABD}. Because all relatively hierarchically hyperbolic groups satisfy the bounded domain dichotomy, this result yields a relative HHS structure with relatively unbounded products for the Cayley graph of any relative HHG with clean containers, which is a useful tool in its own right, and will play a central role in the proof that graph products of infinite Morse local-to-global groups are Morse local-to-global in Section \ref{sec MLTG}.

\subsection{Active Subpaths}
\label{subsec active}

One tool we will utilize regarding hierarchy paths is the fact that they have ``active subpaths" for relevant domains. That is, if the endpoints of a hierarchy path have sufficiently large projection to $CU$, then the hierarchy path has a ``long" subpath contained in a uniform neighborhood of the product region $\mathbf{P}_U$. This statement was originally published as \cite[Proposition 5.17]{behrstock2019hierarchically}, however, that statement contained an error. A corrected version of the statement appears as \cite[Proposition 4.24]{russell2023convexity}. A slightly modified version of the corrected statement also appears as \cite[Proposition 20.1]{casals2022real}. Moreover, \cite{casals2022real} discuss the instances where the incorrect version of the statement was used in the literature, and how these cases are rectified.

The following proposition generalizes the active subpath property of hierarchy paths to the case of a relative HHS. The argument is the same as that of \cite[Proposition 20.1]{casals2022real}, but it is reproduced in full detail here both to show that the construction satisfies the third bullet point (which is not directly stated in \cite[Proposition 20.1]{casals2022real}) and to demonstrate that the argument at no point requires the hyperbolicity of the $\sqsubseteq$-minimal geodesic spaces. We also take this opportunity to add details to the argument. The third bullet point will be important later in the paper, so we state it explicitly.

\begin{prop}
\label{active subpaths}
Let $(\mathcal{X},\mathfrak{S})$ be a relative HHS. For all $\lambda\geq 1$, there exists $\nu_\lambda$ such that the following holds. Let $x,y\in\mathcal{X}$, let $\gamma$ be a $(\lambda,\lambda)$-hierarchy path from $x$ to $y$, and let $U\in\mathfrak{S}$ be $200\lambda E$-relevant for the points $x$ and $y$. Then $\gamma$ has a subpath $\beta$ such that 
\begin{itemize}
    \item $\beta\subset\mathcal{N}_{\nu_\lambda}(\mathbf{P}_U)$,
    \item $\pi_U$ is $\nu_\lambda$-coarsely constant on any subpath of $\gamma$ disjoint from $\beta$, and
    \item $\text{\normalfont{diam}}_U(\pi_U(\beta))\geq d_U(\pi_U(x),\pi_U(y))-24(\lambda E+E)$.
\end{itemize}
\end{prop}

\begin{proof} Without loss of generality let $\gamma\colon\{0,\dots,n\}\to \mathcal{X}$ be a $2\lambda$-discrete path and let $x_i=\gamma(i)$ for $0\leq i\leq n$, so that $d_{\mathcal{X}}(x_i,x_{i+1})\leq 2\lambda$ for all $i$.

Because the projection map $\pi_U$ is $(E,E)$-coarsely Lipschitz, for all $i$, we have
$$d_{U}(\pi_U(x_i),\pi_U(x_{i+1}))\leq E\cdot d_{\mathcal{X}}(x_i,x_{i+1})+E \leq 2\lambda E+E.$$
Since $U$ is $200\lambda E$-relevant for $x$ and $y$, there exist indices $i,i'$ such that $0<i<i'<n$ and
\begin{itemize}
    \item $i$ is minimal with the property that $d_U(\pi_U(x_0),\pi_U(x_i)) > 10(\lambda E+E);$ 
    \item $i'$ is maximal with the property that $d_U(\pi_U(x_{i'}) ,\pi_U(x_n)) > 10(\lambda E+E).$ 
\end{itemize}
We will now bound $d_{\mathcal{X}}(x_i,\mathbf{P}_U)$ and $d_{\mathcal{X}}(x_{i'},\mathbf{P}_U)$, so that $\beta=\gamma|_{[i,i']}$. These distances will be estimated using the distance formula, which states
\begin{equation}
\label{dist_hagen}
d_{\mathcal{X}}(x_i,\mathbf{P}_U)\asymp_{k,c} \sum_{V\in\mathfrak{S}}\{\!\{d_V(\pi_V(x_i),\pi_V(\mathbf{P}_U))\}\!\}_{s_0+\mu},
\end{equation}
where $s_0$ is the distance formula threshold for $(\mathcal{X},\mathfrak{S})$ and $\mu$ is the gate map constant from Proposition \ref{russell gates}. If $V\sqsubseteq U$ or $U\perp V$, then by Proposition \ref{russell gates},
$$ d_V(\pi_V(x_i),\pi_V(\mathbf{P}_U))\leq d_V(\pi_V(x_i),\pi_V(\mathfrak{g}_{\mathbf{P}_U}(x_i)))\leq \mu.$$
Summands in \eqref{dist_hagen} will be nonzero only for domains $V\in\mathfrak{S}$ with $V\sqsupsetneq U$ or $V\pitchfork U$.
Let $C$ be the constant from Proposition \ref{prod map} for $U$.  If $U\sqsubsetneq V$ or $U\pitchfork V$, then $\pi_V(\mathbf{P}_U)\subseteq \mathcal{N}_C(\rho^U_V)$. Therefore consider the following two cases.

Case 1: suppose $U\sqsubsetneq V$. By construction, $$d_U(\pi_U(x_0),\pi_U(x_i)) > 10(\lambda E+E)>E.$$
Let $\alpha$ be a geodesic in $CV$ from $\pi_V(x_0)$ to $\pi_V(x_i)$. By the bounded geodesic image axiom for $(\mathcal{X},\mathfrak{S})$, there exists a point $a\in\alpha$ such that $d_V(a,\rho^U_V)\leq E$. Let $M\colon[1,\infty)\times [0,\infty)\to[0,\infty)$ be the Morse gauge for $\alpha$ in $CV$. Because $\gamma$ is a $(\lambda,\lambda)$-hierarchy path, $\pi_V(\gamma|_{[0,i]})$ is an unparametrized $(\lambda,\lambda)$-quasi-geodesic with endpoints on $\alpha$. In particular, a subpath of $\gamma$ (up to a reparametrization) is a $(\lambda,\lambda)$-quasi-geodesic with endpoints on $\alpha$, so there exists an integer $j\in[0,i]$ such that $d_V(\pi_V(x_j),a)\leq M(\lambda,\lambda)$, which further implies $d_V(\pi_V(x_j),\rho^U_V)\leq E+M(\lambda,\lambda)$. Similarly, there exists $j'\in[i',n]$ such that $d_V(\pi_V(x_{j'}),\rho^U_V)\leq E+M(\lambda,\lambda)$. By definition, diam$(\rho^U_V)\leq E$, so $d_V(\pi_V(x_j),\pi_V(x_{j'}))\leq 3E+2M(\lambda,\lambda)$.

Next, because $\gamma$ is a $(\lambda,\lambda)$-hierarchy path, $\pi_V(\gamma|_{[j,j']})$ is contained in the $\lambda$-neighborhood of a geodesic $\eta$ connecting $\pi_V(x_j)$ to $\pi_V(x_{j'}$). In particular, there exists a point $q\in\eta$ such that $d_V(\pi_V(x_i),q)\leq \lambda$. Thus
\begin{align*}
d_V(\pi_V(x_i),\rho^U_V) &\leq d_V(\pi_V(x_j),\pi_V(x_i)) +d_V(\pi_V(x_j),\rho^U_V)\\
&\leq d_V(\pi_V(x_j),\pi_V(x_i))+E+M(\lambda,\lambda)\\
&\leq d_V(\pi_V(x_j),q)+d_V(\pi_V(x_i),q)+E+M(\lambda,\lambda)\\
&\leq d_V(\pi_V(x_j),\pi_V(x_{j'}))+\lambda+ E+M(\lambda,\lambda)\\
&\leq 3E+2M(\lambda,\lambda)+ \lambda+E+M(\lambda,\lambda)\\
& = 4E +\lambda + 3M(\lambda,\lambda).
\end{align*}
By an identical argument, $d_V(x_{i'},\rho^U_V)\leq 4E+\lambda+3M(\lambda,\lambda)$. Letting $K=4E+\lambda+3M(\lambda,\lambda)$, we have $d_V(\rho^U_V,\pi_V(x_i))\leq K$ and $d_V(\rho^U_V,\pi_V(x_{i'}))\leq K$, as desired.

Case 2: suppose $U\pitchfork V$. There are two sub-cases, depending on whether $d_V(\pi_V(x_0),\pi_V(x_n))$ is greater than $3E$.

Case 2a: suppose $d_V(\pi_V(x_0),\pi_V(x_n))>3E$. By the consistency axiom, either $d_V(\pi_V(x_0),\rho^U_V)\leq E$ or $d_U(\pi_U(x_0),\rho^V_U)\leq E$. Consider the case that $d_V(\pi_V(x_0),\rho^U_V)\leq E$. Additionally, consistency implies either $d_V(\pi_V(x_n),\rho^U_V)\leq E$ or $d_U(\pi_U(x_n),\rho^V_U)\leq E$. However, $d_V(\pi_V(x_n),\rho^U_V)\leq E$ cannot hold because diam$(\rho^U_V)\leq E$ and so the triangle inequality would imply
$$ d_V(\pi_V(x_0),\pi_V(x_n))\leq d_V(\pi_V(x_0),\rho^U_V) + d_V(\pi_V(x_n),\rho^U_V) + \text{diam}(\rho^U_V)\leq 3E,$$
which contradicts the initial assumption. Therefore, $d_V(\pi_V(x_0),\rho^U_V)\leq E$ and $d_U(\pi_U(x_n),\rho^V_U)\leq E$. Because $i$ is minimal such that $d_U(\pi_U(x_0),\pi_U(x_i)) > 10(\lambda E+E)$, we have
$$ d_U(\pi_U(x_0),\pi_U(x_{i-1})) \leq 10(\lambda E+E),$$
so by the triangle inequality
\begin{align}
d_U(\pi_U(x_0),\pi_U(x_i))&\leq d_U(\pi_U(x_0),\pi_U(x_{i-1}))+d_U(\pi(x_i),\pi(x_{i-1})) \nonumber \\
&\leq 10(\lambda E+E)+d_U(\pi(x_i),\pi(x_{i-1})) \nonumber \\
&\leq 10(\lambda E+E)+E\cdot d_{\mathcal{X}}(x_i,x_{i-1})+E \nonumber \\
&\leq 10(\lambda E+E)+E\cdot 2\lambda+E \nonumber \\
& = 12\lambda E+11E \nonumber \\
&\leq 12(\lambda E+E). \label{bullet 3}
\end{align}
Similarly, $d_U(\pi_U(x_{i'}),\pi_U(x_n))\leq 12(\lambda E+E)$. Additionally
$$d_U(\pi_U(x_0),\rho^V_U)+\text{diam}(\rho^V_U) +d_U(\pi_U(x_n),\rho^V_U)\geq d_U(\pi_U(x_0),\pi_U(x_n))>200\lambda E,$$
which implies that
$$ d_U(\pi_U(x_0),\rho^V_U) > 200\lambda E-E-E = 200\lambda E-2E,$$
so
$$ d_U(\pi_U(x_i),\rho^V_U) +d_U(\pi_U(x_0),\pi_U(x_i)) \geq d_U(\pi_U(x_0),\rho^V_U)>200\lambda E-2E,$$
which further implies
$$d_U(\pi_U(x_i),\rho^V_U)> 200\lambda E-2E - 12(\lambda E+E)=188\lambda E-14E>E.$$
Thus by consistency, $d_V(\pi_V(x_i),\rho^U_V)\leq E$. Moreover,
$$ d_U(\pi_U(x_{i'}),\pi_U(x_n))>10(\lambda E+E) \qquad\text{and}\qquad d_U(\pi_U(x_n),\rho^V_U)\leq E$$
together imply
$$ d_U(\pi_U(x_{i'}),\rho^V_U)>10(\lambda E+E)-E>E.$$
Therefore, consistency implies $d_V(\pi_V(x_{i'}),\rho^U_V)\leq E$. If consistency had originally given $d_U(\pi_U(x_0),\rho^V_U)\leq E$, then a symmetric argument would imply $d_V(\pi_V(x_i),\rho^U_V)\leq E$ and $d_V(\pi_V(x_{i'}),\rho^U_V)\leq E$.

Case 2b: let $U\pitchfork V$ and suppose $d_V(\pi_V(x_0),\pi_V(x_n))\leq 3E$. Observe that by the argument above, $d_U(\pi_U(x_0),\pi_U(x_i))\leq 12(\lambda E+E)$ and $d_U(\pi_U(x_{i'}),\pi_U(x_n))\leq 12(\lambda E+E)$. Then because $d_U(\pi_U(x_0),\pi_U(x_n))>200\lambda E$,
$$ d_U(\pi_U(x_i),\pi_U(x_{i'}))>200\lambda E-24(\lambda E+E).$$
Moreover, using a similar argument as in Case 2a, if $d_U(\pi_U(x_i),\rho^V_U)\leq E$ and $d_U(\pi_U(x_{i'}),\rho^V_U)\leq E$, then
$$ d_U(\pi_U(x_i),\pi_U(x_{i'}))\leq d_U(\pi_U(x_i),\rho^V_U)+d_U(\pi_U(x_{i'}),\rho^V_U)+\text{diam}(\rho^V_U)\leq 3E,$$
which contradicts $d_U(\pi_U(x_i),\pi_U(x_{i'}))>200\lambda E-24(\lambda E+E)$. Therefore, at least one of $\pi_U(x_i),\pi_U(x_{i'})$ is distance $>E$ from $\rho^V_U$. Without loss of generality, let $d_U(\pi_U(x_i),\rho^V_U)> E$. Then consistency implies $d_V(\pi_V(x_i),\rho^U_V)\leq E$. As $\gamma$ is a $(\lambda,\lambda)$-hierarchy path, $\pi_V(\gamma)$ is contained in the $\lambda$-neighborhood of a geodesic $\zeta$ connecting $\pi_V(x_0)$ to $\pi_V(x_n)$. Thus there exist points $z,w\in\zeta$ such that $d_V(\pi_V(x_i),z)\leq \lambda$ and $d_V(\pi_V(x_{i'}),w)\leq \lambda$. Using the triangle inequality,
\begin{align*}
d_V(\pi_V(x_{i'}),\rho^U_V)&\leq d_V(\pi_V(x_i),\pi_V(x_{i'})) + d_V(\pi_V(x_i),\rho^U_V)\\
&\leq d_V(\pi_V(x_i),\pi_V(x_{i'})) + E\\
&\leq d_V(\pi_V(x_i),z) + d_V(z,\pi_V(x_{i'})) + E\\
&\leq \lambda + d_V(z,\pi_V(x_{i'})) + E\\
&\leq \lambda + d_V(z,w) + d_V(w,\pi_V(x_{i'})) + E\\
&\leq \lambda + d_V(\pi_V(x_0),\pi_V(x_n)) + \lambda + E\\
&\leq 3E + 2\lambda + E\\
&= 4E + 2\lambda.
\end{align*}
Combining the results of Cases 1 and 2, there exists a constant $K'=K'(\lambda,E)$ such that if  $V\pitchfork U$ or $V\sqsupsetneq U$,
\begin{equation}
\label{active subpaths full K}
d_V(\pi_V(x_i),\rho^U_V),d_V(\pi_V(x_{i'}),\rho^U_V)\leq K'.
\end{equation} 
Now set $\nu_1 = \kappa^{\times}(K')+C$, where $\kappa^\times$ is a function depending on relative HHS constants such that $\mathbf{P}_U,\mathbf{E}_U$, and $\mathbf{F}_U$ are $\kappa^\times$-hierarchically quasi-convex. Thus, \eqref{active subpaths full K} and $\pi_V(\mathbf{P}_U)\subseteq \mathcal{N}_C(\rho^U_V)$ together ensure that $x_i,x_{i'}\in\mathcal{N}_{\nu_1}(\mathbf{P}_U)$. The hierarchical quasi-convexity of $\mathbf{P}_U$ along with the fact that $\gamma$ is a $(\lambda,\lambda)$-hierarchy path implies that $\nu_1$ can be increased by an amount depending only on $\lambda$ and the relative HHS constants to yield $\nu_{\lambda}$ such that $x_j\in \mathcal{N}_{\nu_{\lambda}}(\mathbf{P}_U)$ for $i\leq j\leq i'$.

Letting $\beta=\gamma|_{i,\dots,i'}$, there exists a $\nu_{\lambda}$ such that $\beta\subset\mathcal{N}_{\nu_{\lambda}}(\mathbf{P}_U)$. This completes the proof of the first bullet point. For the second bullet point, note that for $j<i$ and $j'>i'$, 
$$d_U(\pi_U(x_0),\pi_U(x_j))\leq 10(\lambda E+E) \quad\text{and}\quad d_U(\pi_U(x_{j'}),\pi_U(x_n))\leq 10(\lambda E+E),$$ so (by possibly increasing $\nu_{\lambda}$), $\pi_U$ is $\nu_{\lambda}$-coarsely constant on any subpath of $\gamma$ disjoint from $\beta$. For the final bullet point, by \eqref{bullet 3},
\begin{align*}
    \text{diam}_U(\pi_U(\beta))+24(\lambda E+E) &= 
    12(\lambda E+E)+\text{diam}_U(\pi_U(\beta))+12(\lambda E+E)\\
    &\geq
    d_U(\pi_U(x),\pi_U(x_i))+\text{diam}_U(\pi_U(\beta))+d_U(\pi_U(x_i'),\pi_U(y))\\
    &\geq d_U(\pi_U(x),\pi_U(x_i))+d_U(\pi_U(x_i),\pi_U(x_i'))+d_U(\pi_U(x_i'),\pi_U(y))\\
    &\geq d_U(\pi_U(x),\pi_U(y)),\\
    \implies \text{diam}_U(\pi_U(\beta))&\geq d_U(\pi_U(x),\pi_U(y))-24(\lambda E+E),
\end{align*}
completing the proof.
\end{proof}
 
\subsection{Maximization}
\label{section abd}

We begin by modifying the argument of \cite[Theorem 3.7]{abbott2021largest} for the relative HHS case.  Under mild conditions on a relative HHS $(\mathcal{X},\mathfrak{S})$, we will produce a subset of domains $\mathfrak{T}\subseteq\mathfrak{S}$ for which the pair $(\mathcal{X},\mathfrak{T})$ is a relative HHS with the additional property that, roughly speaking, all product regions have unbounded factors. This means the proof will verify that the new pair $(\mathcal{X},\mathfrak{T})$ satisfies all twelve of the relative HHS axioms.

The initial relative HHS will satisfy the mild conditions of having clean containers and the bounded domain dichotomy. It is worth noting that every relative HHG will have the bounded domain dichotomy automatically, which is why it is considered mild. We will later show that graph products have clean containers.

\begin{definition}
Let $(\mathcal{X},\mathfrak{S})$ be a relatively hierarchically hyperbolic space. For each $W\in\mathfrak{S}$ and $U\in\mathfrak{S}_W$ with $\mathfrak{S}_W\cap\mathfrak{S}^{\perp}_U\ne\emptyset$, the container axiom provides a domain $Q\sqsubsetneq W$ such that $V\sqsubseteq Q$ whenever $V\in\mathfrak{S}_W\cap\mathfrak{S}^{\perp}_U$. If, for each $U$, the container is such that $Q\perp U$, then $(\mathcal{X},\mathfrak{S})$ has \emph{clean containers}.
\end{definition}

Note that the assumption of clean containers does not appear in \cite[Theorem 3.7]{abbott2021largest}. However, that theorem relies on \cite[Theorem A.1]{abbott2021largest} in the appendix, whose proof is incorrect. See \cite{abbott2025structure} for a full discussion.

\begin{definition}
A relative HHS $(\mathcal{X},\mathfrak{S})$ has the \emph{$M$-bounded domain dichotomy} if there exists $M>0$ such that any $U\in\mathfrak{S}$ with diam$(CU)>M$ satisfies diam$(CU)=\infty$. If the value for $M$ is not important, then $(\mathcal{X},\mathfrak{S})$ has the \emph{bounded domain dichotomy}.
\end{definition}

We must also define the notion of unbounded products, which is the additional property that the constructed structure possesses. 

\begin{definition}
A relative HHS $(\mathcal{X},\mathfrak{S})$ has \emph{unbounded products} if it has the bounded domain dichotomy and the property that if $U\in\mathfrak{S}-\{S\}$ has $\mathbf{F}_U$ unbounded, then $\mathbf{E}_U$ is also unbounded. A relative HHS has \emph{relatively unbounded products} if it has the bounded domain dichotomy and the property that if $U\in\mathfrak{S}-\{S\}$ has $\mathbf{F}_U$ unbounded and is not $\sqsubseteq$-minimal, then $\mathbf{E}_U$ is also unbounded. A relative HHS has \emph{unbounded minimal products} if every $\sqsubseteq$-minimal $U\in\mathfrak{S}$ with $\mathbf{F}_U$ unbounded has $\mathbf{E}_U$ also unbounded.
\end{definition}

The previous definition contains three different, albeit related, notions. Theorem \ref{clean containers ABD} will produce a relative HHS structure with relatively unbounded products. We show that graph products of infinite groups with no isolated vertices have unbounded minimal products in Corollary \ref{multivertex}. Relatively unbounded products together with unbounded minimal products imply genuine unbounded products, which is a necessary step in Theorem \ref{abd6.2}, which shows that the top level space of the new structure is a Morse detectability space.

The construction that we are undertaking will generate a hierarchy structure for which every domain that is not $\sqsubseteq$-maximal nor $\sqsubseteq$-minimal has a product region with unbounded factors. We begin by isolating the subset of domains for which this is true, and show that it is closed under nesting.

\begin{definition}
\label{Sm}
Let $(\mathcal{X},\mathfrak{S})$ be a relative HHS. Let $M>0$, and define $\mathfrak{S}^M\subset\mathfrak{S}$ to be the set of domains $U\in\mathfrak{S}$ such that there exists $V\in \mathfrak{S}$ and $W\in\mathfrak{S}^\perp_V$ satisfying: 
\begin{itemize}
    \item $U\sqsubseteq V$ 
    \item diam$(CV)>M$
    \item diam$(CW)>M$.
\end{itemize}
\end{definition}

\begin{definition}
A set $\mathfrak{U}\subset \mathfrak{S}$ is \emph{closed under nesting} if whenever $U\in \mathfrak{U}$ and $V\sqsubseteq U$, then $V\in\mathfrak{U}$.
\end{definition}

\begin{lemma}
\label{GMclosed}
For any $M>0$, the set $\mathfrak{S}^M$ is closed under nesting.
\end{lemma}

\begin{proof}
The argument in \cite[Lemma 3.1]{abbott2021largest} goes through verbatim.
\end{proof}

The next proposition shows that given a relative HHS, the nested partial tuples themselves can be given the structure of a relative HHS for the appropriate domains. It is a key component of the verification of the Large Links axiom (10) in Theorem \ref{clean containers ABD}. This proposition is based on \cite[Proposition 5.11]{behrstock2019hierarchically}, which proves the result for any HHS.

\begin{prop}
\label{BHS19 5.11}
Let $(\mathcal{X},\mathfrak{S})$ be an $E$-relative HHS. There exists a constant $F$, depending only on $E$, such that for any $U\in\mathfrak{S}$ and any slice $\mathbf{F}_U\times\{\Vec{e}\}\subset \mathcal{X}$ endowed with the subspace metric, the space $(\mathbf{F}_U,\mathfrak{S}_U)$ is an $F$-relative HHS.
\end{prop}

\begin{proof}
The proof goes through verbatim as in \cite[Proposition 5.11]{behrstock2019hierarchically}.
\end{proof}

The following lemmas and proposition build towards Lemma \ref{bgi}, which in turn is used in the proof of the Bounded Geodesic Image axiom (9) in Theorem \ref{clean containers ABD}. Lemma \ref{lem7A}, Lemma \ref{bhs 2,14}, and Proposition \ref{redone} are slight generalizations of \cite[Lemma 2.1]{behrstock2019hierarchically}, \cite[Lemma 2.14]{behrstock2019hierarchically}, and \cite[Proposition 2.4]{Behrstock_2017}, respectively, which were proven for the non-relative case.

\begin{lemma}
\label{lem7A}
Let $(\mathcal{X},\mathfrak{S})$ be a relative HHS and let $U_1,U_2,...,U_k\in\mathfrak{S}$ be pairwise orthogonal. Then $k\leq E$.
\end{lemma}

\begin{proof} The argument in \cite[Lemma 2.1]{behrstock2019hierarchically} goes through verbatim.
\end{proof}

\begin{definition}
Let $(\mathcal{X},\mathfrak{S})$ be a relative HHS. The \emph{level} $\ell_U$ is defined inductively as follows. If $U$ is $\sqsubseteq$-minimal, then $\ell_U=1$. For any non-$\sqsubseteq$-minimal element $V,\in\mathfrak{S}$, $\ell_V=k+1$ if $k$ is the maximal integer such that there exists a $W\sqsubsetneq V$ with $\ell_W=k$. Moreover, define $\mathfrak{T}^\ell_U$ to be the set of $V\in\mathfrak{S}_U$ such that $\ell_U-\ell_V=\ell$. 
\end{definition}

\begin{lemma}
\label{bhs 2,14}
Given a relative HHS $(\mathcal{X},\mathfrak{S})$, let $\chi$ be the maximum cardinality of a set of pairwise orthogonal elements of $\mathfrak{T}^{\ell}_U$. Then there exists a $\chi$-coloring of the set of relevant elements of $\mathfrak{T}^{\ell}_U$ such that non-transverse elements have different colors.
\end{lemma}

\begin{proof}
The argument in \cite[Lemma 2.14]{behrstock2019hierarchically} goes through verbatim because it does not utilize the hyperbolicity of $\sqsubseteq$-minimal elements at any point.
\end{proof}

\begin{prop}
\label{redone}
Fix a relative HHS, $(\mathcal{X},\mathfrak{S})$, and let $\mathfrak{U}\subset \mathfrak{S}$ be closed under nesting. The space $(\widehat{\mathcal{X}}_{\mathfrak{U}}, \mathfrak{S} - \mathfrak{U})$ is a relative HHS, where the associated $C(*)$, $\pi_*$, $\rho^*_*$, $\sqsubseteq$, $\perp$, $\pitchfork$ are the same as in the original structure.
\end{prop}

\begin{proof}
The proof goes through verbatim as in \cite[Proposition 2.4]{Behrstock_2017}, noting that \cite[Lemma 2.8 (Uniqueness)]{Behrstock_2017} requires Active Subpaths (Proposition \ref{active subpaths}), as well as condition (3) in the definition of a hierarchy path (Definition \ref{hdef}), which we prove for the relative case in Proposition \ref{hpaths}.
\end{proof}

As with Lemma \ref{bhs 2,14} and Proposition \ref{redone}, Lemma \ref{bgi} is a modified version of \cite[Lemma 3.6]{abbott2021largest} for the case of a relative HHS. The argument is the same, but is worked through in its entiretly for completeness and to demonstrate that it does not require $\sqsubseteq$-minimal geodesic spaces to be hyperbolic.

\begin{lemma}
\label{bgi}
Let $(\mathcal{X},\mathfrak{S})$ be a relative HHS, and consider a set $\mathfrak{T}\subset \mathfrak{S}$ that is closed under nesting. Let $\lambda\geq 1$, and let $\nu_{\lambda}$ be the associated constant from Proposition \ref{active subpaths}. For any $x,y\in\mathcal{X}$, let $\gamma$ be a $(\lambda,\lambda)$-hierarchy path in $(\mathcal{X},\mathfrak{S})$ connecting $x$ and $y$. Then the path obtained by including $\gamma\subset\mathcal{X}\subset\hat{\mathcal{X}}_{\mathfrak{T}}$ is an unparametrized quasi-geodesic. Moreover, if for each $W\in\mathfrak{T}$ which is a $200\lambda E$-relevant domain for $x$ and $y$, and each $e\in\mathbf{E}_W$, we modify $\gamma$ by removing all but the first and last vertices contained in the $\nu_{\lambda}$-neighborhood of $\mathbf{F}_W\times \{e\}$, then the new path $\hat{\gamma}$ is a hierarchy path for $(\hat{\mathcal{X}}_{\mathfrak{T}},\mathfrak{S}-\mathfrak{T})$. 
\end{lemma}

\begin{proof}
The proof is by induction on complexity. Fix $x,y\in\mathcal{X}$ as well as a $(\lambda,\lambda)$-hierarchy path connecting them. Consider all the $\sqsubseteq$-minimal elements $\mathfrak{U}\subset\mathfrak{T}$ which are $200\lambda E$-relevant for $x$ and $y$. By Proposition \ref{active subpaths}, for each $U\in\mathfrak{U}$, there exists a subpath $\beta_U\subset\gamma$ which is contained in the $\nu_{\lambda}$-neighborhood of $\mathbf{P}_U$. By the definition of $\mathbf{P}_U$, $\beta_U$ passes through the $\nu_{\lambda}$-neighborhood of a collection of slices $\mathbf{F}_U\times\{\Vec{e}\}$. Next, consider $\mathfrak{T}^{\ell_S-1}_S\subset\mathfrak{S}$. By Lemma \ref{lem7A}, the maximum cardinality of a set of pairwise orthogonal elements of $\mathfrak{T}^{\ell_S-1}_S$ is $E$. By Lemma \ref{bhs 2,14} there exists an $E$-coloring of the set of relevant elements of $\mathfrak{T}^{\ell_S-1}_S$ such that non-transverse elements have different colors. Since $\mathfrak{U}\subseteq \mathfrak{T}^{\ell_S-1}_S$ and every element of $\mathfrak{U}$ is $200\lambda E$-relevant, there exists an $E$-coloring of $\mathfrak{U}$ such that all domains of a particular color are pairwise transverse.

Starting from $(\mathcal{X},\mathfrak{S})$, we proceed one color at a time. For the first color $c_1$, all domains of that color are $\sqsubseteq$-minimal by definition, so the set of domains of color $c_1$ is closed under nesting. Create the factored space by coning off those domains. By Proposition \ref{redone}, this factored space is a relative HHS with the property that, in this space, the $\sqsubseteq$-minimal $200\lambda E$-relevant domains for $x$ and $y$ are exactly the original ones except for the ones we have coned off thus far. Since this path still travels monotonically through the $\nu_{\lambda}$ neighborhood of some slices $\mathbf{F}_U\times\{\Vec{e}\}$ of the product regions of the $200\lambda E$-relevant domains, it is an unparametrized quasi-geodesic in this new factored space. Thus the path $\hat{\gamma}$ is a quasi-geodesic and thus a $(C,C)$-hierarchy path in the new factored space, with $C$ depending only on $\lambda$, $\mathfrak{T}$, and the hierarchy constants from $(\mathcal{X},\mathfrak{S})$. Once the colors of $\mathfrak{U}$ are exhausted, repeat one step up the nesting lattice. Since both the complexity of a $(\mathcal{X},\mathfrak{S})$ and the colorings are bounded by $E$, this will terminate in at most $E^2$ steps. Finally, cone off any domains in $\mathfrak{T}$ which are not relevant for $x$ and $y$ to obtain the space $(\hat{\mathcal{X}}_{\mathfrak{T}},\mathfrak{S}-\mathfrak{T})$. Through the final step, $\hat{\gamma}$ remains a uniform quality hierarchy path since it is still a quasi-geodesic.
\end{proof}

We are now ready to construct the new relative HHS structure. The proof follows \cite[Theorem 3.7]{abbott2021largest}, but we take this opportunity to fill in some missing details and correct some minor errors in the proof.

\begin{thm}
\label{clean containers ABD}
Every relatively hierarchically hyperbolic space with the bounded domain dichotomy and clean containers admits a relatively hierarchically hyperbolic structure with relatively unbounded products.
\end{thm}

\begin{proof} Let $(\mathcal{X},\mathfrak{S})$ be an $E$-relatively hierarchically hyperbolic space with the $M$-bounded domain dichotomy. Without loss of generality, let $E\geq M$. Thus, any finite diameter associated geodesic space in $(\mathcal{X},\mathfrak{S})$ is $E$-hyperbolic. Let $\mathfrak{T}$ denote the $\sqsubseteq$-maximal element $S$ together with $\sqsubseteq$-minimal elements of $\mathfrak{S}$ with infinite-diameter geodesic spaces, and the subset of $\mathfrak{S}$ consisting of all $U\in\mathfrak{S}$ with both $\mathbf{F}_U$ and $\mathbf{E}_U$ unbounded. For convenience, define 
$$\text{Unb} = \{U\in\mathfrak{T}\;|\;\text{both $\mathbf{F}_U$ and $\mathbf{E}_U$ are unbounded}\},$$
$$\text{Min} = \{U\in\mathfrak{T}\;|\;\text{$U$ is $\sqsubseteq$-minimal, $\mathbf{E}_U$ is bounded, and $\mathbf{F}_U$ is unbounded}\}.$$ So,
$$\mathfrak{T}=\{S\}\sqcup \text{Unb}\sqcup \text{Min}.$$

We begin to define our new relatively hierarchically hyperbolic structure on $\mathcal{X}$ by taking $\mathfrak{T}$ as our index set. For each $U\in \mathfrak{T}-\{S\}$ we set the associated geodesic space $\mathcal{T}_U$ to be $CU$. Note that for non-$\sqsubseteq$-minimal $U\in\mathfrak{T}$, the space $\mathcal{T}_U$ will be hyperbolic because $CU$ is hyperbolic.

For the top-level domain, $S$, we obtain a hyperbolic space, $\mathcal{T}_S$ as follows. By Lemma \ref{GMclosed}, $\mathfrak{S}^M$ is closed under nesting, so
$$\mathfrak{S}^{M+} = \mathfrak{S}^M\cup\{U\in\mathfrak{S}\;|\;U\text{ is $\sqsubseteq$-minimal with unbounded $\mathbf{F}_U$}\}$$
is also closed under nesting. By \cite[Proposition 2.4]{Behrstock_2017}, $\widehat{\mathcal{X}}_{\mathfrak{S}^{M+}}$ is a hierarchically hyperbolic space with index set $\mathfrak{S}-\mathfrak{S}^{M+}$, because $\mathfrak{S}^{M+}$ is closed under nesting and contains all $\sqsubseteq$-minimal domains which are not $E$-hyperbolic. Fix any orthogonal $U,V\in\mathfrak{S}-\mathfrak{S}^{M+}$. If $\text{diam}(CU)>M$ and $\text{diam}(CV)>M$, then by definition $U,V\in\mathfrak{S}^{M+}$, which is a contradiction. Therefore, $\widehat{\mathcal{X}}_{\mathfrak{S}^{M+}}$ has the property that no pair of orthogonal domains both have geodesic spaces of diameter larger than $M$. Thus by \cite[Corollary 2.16]{Behrstock_2017}, it is hyperbolic for some constant depending only on $(\mathcal{X},\mathfrak{S})$, and $M$. We set $\mathcal{T}_S = \widehat{\mathcal{X}}_{\mathfrak{S}^{M+}}$.

To avoid confusion, we use the notation $d_S$ for distance in $\mathcal{T}_S$ and the notation $d_{CS}$ for the distance in $CS$. Moreover, $\pi_{CS}$ denotes the projection from $\mathcal{X}$ to $CS$, whereas $\pi_S$ denotes the projection from $\mathcal{X}$ to $\mathcal{T}_S$.

When $U\ne S$, the projections are as defined in the original relatively hierarchically hyperbolic space. We take the projection $\pi_S$ to be the factor map $\mathcal{X}\to \mathcal{T}_S$. If $U\in\mathfrak{T}$ and $U\ne S$, then the relative projections are defined as in $(\mathcal{X},\mathfrak{S})$. For the remaining case the relative projection $\rho_S^V$ is defined to be the image of $\mathbf{F}_V$ under the factor map $\mathcal{X}\to \mathcal{T}_S$.\\

\noindent
\textbf{(1) Projections:} The only case to check is for the top-level domain $S$. Since $\pi_S$ is a factor map, for all $x\in\mathcal{X}$, $\pi_S(x)\ne\emptyset$ and diam$(\pi_S(x))=0$. Moreover, $\pi_S$ is $(1,0)$-coarsely Lipschitz and $\mathcal{T}_S = \hat{\mathcal{X}}_{\mathfrak{S}^{M+}} \subseteq \mathcal{N}_2(\pi_S(\mathcal{X}))$, so this axiom is satisfied.\\

\noindent
\textbf{(2) Nesting:} The partial order is inherited from $(\mathcal{X},\mathfrak{S})$. The projections are given in the construction. The diameter bound is inherited from $(\mathcal{X},\mathfrak{S})$ except for the case of $\rho_S^V$ for $V\in\mathfrak{T}$. By construction, $\rho_S^V$ is the image of $\mathbf{F}_V$ under the factor map $\mathcal{X}\to \mathcal{T}_S$, and $\mathfrak{T}\subset\mathfrak{S}^{M+}$, thus the diameter of $\rho_S^V$ is bounded.

\noindent\\
\textbf{(3) Orthogonality:} This axiom only involves domains which are not $\sqsubseteq$-maximal, hence it is inherited from $(\mathcal{X},\mathfrak{S})$ by construction.

\noindent\\
\textbf{(4) Transversality:} This axiom only involves domains which are not $\sqsubseteq$-maximal, hence it is inherited from $(\mathcal{X},\mathfrak{S})$ by construction.

\noindent\\
\textbf{(5) Hyperbolicity:} As shown above, $\mathcal{T}_S$ is hyperbolic, and $CU$ is hyperbolic for any $U\in\mathfrak{T}-\{S\}$ which is not $\sqsubseteq$-minimal by construction.

\noindent\\
\textbf{(6) Finite Complexity:} This follows directly from the fact that $\mathfrak{T}\subseteq\mathfrak{S}$ and the partial order on $\mathfrak{T}$ is inherited from $\mathfrak{S}$.

\noindent\\
\textbf{(7) Containers:} Fix $W\in\mathfrak{T}$ and $U\in\mathfrak{T}_W$ with $\mathfrak{T}_W\cap\mathfrak{T}^{\perp}_U\ne\emptyset$. Next fix $V\in\mathfrak{T}_W\cap\mathfrak{T}^{\perp}_U$. By the container axiom for $(\mathcal{X},\mathfrak{S})$, there exists a domain $Q\sqsubsetneq W$ such that $V\sqsubseteq Q$. Moreover, $(\mathcal{X},\mathfrak{S})$ has clean containers, so $Q\perp U$. The domains $U$ and $V$ are contained in $\mathfrak{T}=\{S\}\sqcup\text{Unb}\sqcup\text{Min}$, and neither is $S$ because they are orthogonal to each other. By construction of Unb and Min, the spaces $\mathbf{F}_U$ and $\mathbf{F}_V$ are unbounded. Because $V\sqsubseteq Q$ and $U\perp Q$, the spaces $\mathbf{F}_Q$ and $\mathbf{E}_Q$ are unbounded. Thus $Q\in\text{Unb}\subset\mathfrak{T}$, so this axiom is satisfied.\\

\noindent
\textbf{(8) Uniqueness:} By \cite[Corollary 2.9]{Behrstock_2017}, there exists a map
$$f\colon\hat{\mathcal{X}}_{\mathfrak{S}-\{S\}}\to CS$$
which is a $(C,C)$-quasi-isometry.  Fix $r>0$, and let 
$$r'=Cr+C^2+M+E'r+E^2,$$
where $E'$ is the hierarchy constant of $(\mathcal{T}_S,\mathfrak{S}-\mathfrak{S}^{M+})$.
Define $\theta'(r)=\theta(r')$, where $\theta$ is the uniqueness function for $(\mathcal{X},\mathfrak{S})$. Fix $x,y\in\mathcal{X}$ such that $\text{d}_{\mathcal{X}}(x,y)\geq\theta'(r)$. By the uniqueness axiom for $(\mathcal{X},\mathfrak{S})$, there exists $U\in\mathfrak{S}$ such that $\text{d}_{CU}(\pi_{CU}(x),\pi_{CU}(y))\geq r' \geq r$. 

If $U\in\mathfrak{T}-\{S\}$, then the axiom is satisfied because $\mathcal{T}_U=CU$, so suppose $U=S$. Observe that $\pi_S\colon\mathcal{X}\to\mathcal{T}_S$ is 1-Lipschitz and there is a 1-Lipschitz map $g\colon\mathcal{T}_S\to \hat{\mathcal{X}}_{\mathfrak{S}-\{S\}}$. Moreover, by the construction in the proof of \cite[Corollary 2.9]{Behrstock_2017}, the following diagram commutes:

\[\begin{tikzcd}
	{\mathcal{X}} \\
	{\mathcal{T}_S} \\
	{\hat{\mathcal{X}}_{\mathfrak{S}-\{S\}}} & {\pi_{CS}(\mathcal{X})\subseteq CS}
	\arrow["{\pi_S}"', from=1-1, to=2-1]
	\arrow["{\pi_{CS}}", from=1-1, to=3-2]
	\arrow["g"', from=2-1, to=3-1]
	\arrow["f"', from=3-1, to=3-2]
\end{tikzcd}\]
Therefore,
\begin{align*}
d_S(\pi_S(x),\pi_S(y))&\geq d_{\hat{\mathcal{X}}_{\mathfrak{S}-\{S\}}}(g\circ\pi_S(x),g\circ\pi_S(y))\\
& \geq \frac{1}{C}\text{d}_{CS}(f\circ g\circ\pi_S(x),f\circ g\circ\pi_S(y))-C\\
& = \frac{1}{C}\text{d}_{CS}(\pi_{CS}(x),\pi_{CS}(y))-C\\
& \geq \frac{r'}{C}-C\\
&\geq r.
\end{align*}
so the axiom is satisfied in this case.

Finally suppose $U\in\mathfrak{S}-\mathfrak{T}$ (Note that this approach is different than that of \cite{abbott2021largest}). There are two possibilities: $U\in\mathfrak{S}^{M+}$ or $U\in\mathfrak{S}-\mathfrak{S}^{M+}$. If $U\in\mathfrak{S}^{M+}$, then $\mathbf{E}_U$ is unbounded by definition. Since $U\notin\mathfrak{T}$, it must be the case that $\mathbf{F}_U$ is bounded, and, in particular, $CU$ must be bounded. However, $(\mathcal{X},\mathfrak{S})$ has the $M$-bounded domain dichotomy, and $d_U(\pi_U(x),\pi_U(y))\geq r'>M$, so $CU$ cannot be bounded. Hence, $U\in\mathfrak{S}-\mathfrak{S}^{M+}$. Recall that $(\mathcal{T}_S,\mathfrak{S}-\mathfrak{S}^{M+})$ is an $E'$-HHS where $CU$ and $\pi_U$ are the same as in the original structure. In particular, $\pi'_U\colon\mathcal{T}_S\to CU$ is $(E',E')$-coarsely Lipschitz and $\pi_U=\pi'_U\circ\pi_S$. Thus,
\begin{align*}
    d_S(\pi_S(x),\pi_S(y))&\geq \frac{1}{E'}d_U(\pi'_U\circ\pi_S(x),\pi'_U\circ\pi_S(y))-E'\\
    &=\frac{1}{E'}d_U(\pi_U(x),\pi_U(y))-E'\\
    &\geq \frac{r'}{E'}-E'\\
    &\geq r.
\end{align*}
which satisfies the final case.

\noindent\\
\textbf{(9) Bounded Geodesic Image:} Let $\lambda\geq D$ where $D$ is the constant from Proposition \ref{hpaths}. Additionally, let $\nu_{\lambda}$ be the constant from Proposition \ref{active subpaths} associated to $\lambda$. Because $\mathcal{T}_S$ is hyperbolic, let $M$ be the Morse gauge such that any geodesic in $\mathcal{T}_S$ is $M$-Morse. Finally, let $C$ be the constant coming from Lemma \ref{bgi} such that any $(\lambda,\lambda)$-hierarchy path $\gamma$ in $(\mathcal{X},\mathfrak{S})$ can be modified to be a $(C,C)$-hierarchy path $\hat{\gamma}$ in $(\hat{\mathcal{X}}_{\mathfrak{S}^{M+}},\mathfrak{S}-\mathfrak{S}^{M+}) =(\mathcal{T}_S,\mathfrak{S}-\mathfrak{S}^{M+}) $. Now define 
$$E'=200\lambda E + \nu_{\lambda}+M(C,C).$$
We will show that for all $x,y\in\mathcal{X}$ and $V,W\in\mathfrak{T}$ with $V\sqsubsetneq W$, if 
$$\text{d}_V(\pi_V(x),\pi_V(y))\geq E',$$
then every $CW$-geodesic from $\pi_W(x)$ to $\pi_W(y)$ intersects the $E'$-neighborhood of $\rho^V_W$. If $W\ne S$, this follows from the bounded geodesic image axiom from $(\mathcal{X},\mathfrak{S})$ applied to $V$ and $W$. 

Thus consider the case where $W=S$ and $V\sqsubsetneq S$. Fix $x,y\in\mathcal{X}$ such that $d_V(\pi_V(x),\pi_V(y))\geq E'$, and let  $\gamma$ be a geodesic in $\mathcal{T}_S$ from $\pi_S(x)$ to $\pi_S(y)$. By Proposition \ref{hpaths}, there exists a $(\lambda,\lambda)$-hierarchy path $\alpha$ in $\mathcal{X}$ with endpoints $x$ and $y$. By Lemma \ref{bgi}, there exists $\hat{\alpha}\subset\pi_S(\alpha)$ which is a $(C,C)$-quasi-geodesic for some $C$ depending only on $\lambda$, $\mathfrak{S}^{M+}$, and the hierarchy constants from $(\mathcal{X},\mathfrak{S})$. By Proposition \ref{active subpaths}, $\alpha$ has a subpath $\beta$ such that $\beta\subset\mathcal{N}_{\nu_{\lambda}}(\mathbf{P}_V)$. Because $\hat{\alpha}$ is a $(C,C)$-quasi-geodesic with endpoints on $\gamma$, which is $M$-Morse, for any point $a\in\pi_S(\beta)\subset\hat{\alpha}$, there exists a point $b\in\gamma$ such that $d_S(a,b)\leq M(C,C)$. Moreover, such a point $a\in\pi_S(\beta)\subset\hat{\alpha}$ must exist because by the construction of $\hat{\alpha}$ in Lemma \ref{bgi}, both the initial and final vertex of $\pi_S(\beta)$ are in $\hat{\alpha}$. Therefore, using the fact that $\rho^V_S=\pi_S(\mathbf{F}_V)$, we obtain
\begin{align*}
    d_S(\gamma,\rho^V_S) &\leq d_S(\pi_S(\beta),\rho^V_S) + d_S(\pi_S(\beta),\gamma)\\
    &\leq d_S(\pi_S(\beta),\rho^V_S)+M(C,C)\\
    &=d_S(\pi_S(\beta),\pi_S(\mathbf{F}_V))+M(C,C)\\
    &=d_S(\pi_S(\beta),\pi_S(\mathbf{P}_V))+M(C,C)\\
    &\leq d_{\mathcal{X}}(\beta,\mathbf{P}_V)+M(C,C)\\
    &\leq \nu_{\lambda}+M(C,C)\\
    &\leq E',
\end{align*}
so this axiom is satisfied.

\noindent\\
\textbf{(10) Large Links:} First define the constant
\begin{equation}
\label{E hat large links}
\hat{E} = E'(K_0+1)+ E'(1+ N(1+s_0)+C_0) + M\cdot E,  
\end{equation}
where $E'=\max(E,F)$, $F$ is the relative HHS constant for any $(\mathbf{F}_U,\mathfrak{S}_U)$ from Proposition \ref{BHS19 5.11}, $N$ is defined in the procedure below, and $K_0,C_0,s_0$ come from the distance formula for the HHS $(\mathcal{T}_S,\mathfrak{S}-\mathfrak{S}^{M+})$. We will show that for all $W\in\mathfrak{T}$ and $x,y\in\mathcal{X}$, there exists $\{V_1,\dots,V_m\}\subseteq \mathfrak{T}_W-\{W\}$ such that $m\leq \hat{E}\cdot \text{d}_W(\pi_W(x),\pi_W(y))+\hat{E}$, and for all $U\in\mathfrak{T}_W-\{W\}$, either $U\in\mathfrak{T}_{V_i}$ for some $i$, or $\text{d}_U(\pi_U(x),\pi_U(y))\leq \hat{E}$.

Fix $W\in\mathfrak{T}$ and $x,y\in\mathcal{X}$. If $W$ is $\sqsubseteq$-minimal, then the set $\mathfrak{T}_W-\{W\}=\emptyset$, so the condition is trivially satisfied. Next, suppose $W\in \text{Unb}$. Consider the set $\{T_i\}\subset \mathfrak{S}_W-\{W\}$ provided by the large link axiom for $(\mathcal{X},\mathfrak{S})$. Since $T_i\sqsubsetneq W$ and $\mathbf{E}_W$ is unbounded, $\mathbf{E}_{T_i}$ is unbounded for all $i$. For any $T\in\mathfrak{T}_W-\{W\}$ such that $d_T(\pi_T(x),\pi_T(y))>M\cdot E,$
the large link axiom for $(\mathcal{X},\mathfrak{S})$ implies $$d_T(\pi_T(x),\pi_T(y))=d_{CT}(\pi_T(x),\pi_T(y))>M\cdot E\geq E,$$ so $T\sqsubseteq T_j$ for some $j$. Additionally, by the bounded domain dichotomy property, $d_T(\pi_T(x),\pi_T(y))>M\cdot E\geq M$ implies $\mathbf{F}_T$ is unbounded, so $T\sqsubseteq T_j$ implies $\mathbf{F}_{T_j}$ is unbounded. Thus $T_j\in\mathfrak{T}$. Therefore, the set of such $T_j\in\mathfrak{T}$ has cardinality less than or equal to $|\{T_i\}|$ and so satisfies the appropriate conditions for the large link axiom.

For the final case, let $W=S$; note that this case was not discussed in \cite{abbott2021largest}. Let $A^1_1 = \{T^1_i\}\subseteq \mathfrak{S}-\{S\}$ be the collection of domains provided by the large link axiom for $(\mathcal{X},\mathfrak{S})$. For any $T\in\mathfrak{T}-\{S\}$ 
such that  $ d_T(\pi_T(x),\pi_T(y))>E,$
the large link axiom for $(\mathcal{X},\mathfrak{S})$ implies $T\sqsubseteq T^1_i\in A^1_1$ for some $i$. Define the subset $B^1_1\subseteq A^1_1$ to be all $T^1_i$ such that there exists $T\in \mathfrak{T}$ with $d_T(\pi_T(x),\pi_T(y))>E$ and $T\sqsubseteq T^1_i$. Let $C^1_1= B^1_1\cap\mathfrak{T}$ and let $D^1_1=B^1_1-C^1_1$. If $D^1_1=\emptyset$, then we are done. 

\begin{figure}[t]
\centering
\begin{tikzpicture}
[
    level 1/.style={draw=red},
    level 2/.style={draw=black},
    level 5/.style={draw=red},
    level 6/.style={draw=black},
    level 9/.style={draw=red},
    level 10/.style={draw=black},
]
	\node {$S$}
		child {
                node {$A^1_1$}
                child {
                    node {$B^1_1$}
                    child {node {$C^1_1$}}
                    child {node {$D^1_1$}
                        child {
                            node {$T^1_{i_1}$}
                            child { node {$A^2_{i_1}$}
                                child {node {$B^2_{i_1}$}
                                        child {node {$C^2_{i_1}$}}
                                        child {node {$D^2_{i_1}$}
                                            child {node {$T^2_{j_1}$}
                                                child {node {$A^3_{j_1}$}
                                                child{node {$\cdots$}}}}
                                            child {node {$T^2_{j_2}$}
                                            child {node {$\cdots$}}}
                                            child {node {$\cdots$}}
                                            child {node {$T^2_{j_m}$}
                                            child {node {$\cdots$}}}
                                            }}
                            }}
                        child {
                            node {$T^1_{i_2}$}
                            child {node {$\cdots$}}
                            }
                        child {
                            node {$\cdots$}
                            }
                        child {
                            node {$T^1_{i_n}$}
                            child {node {$\cdots$}}
                            }
                        }
                    }
                };
\end{tikzpicture}
\caption{The procedure of building nesting chains as described in the large link axiom. All black lines represent set containment (identifying the $T^p_q$ with $\{T^p_q\}$). The red lines represent applying the large link axiom to $T^p_q$ domain at the top, and as such every individual domain in the $A^{p+1}_q$ below nests into the $T^p_q$ above it.}
\end{figure}
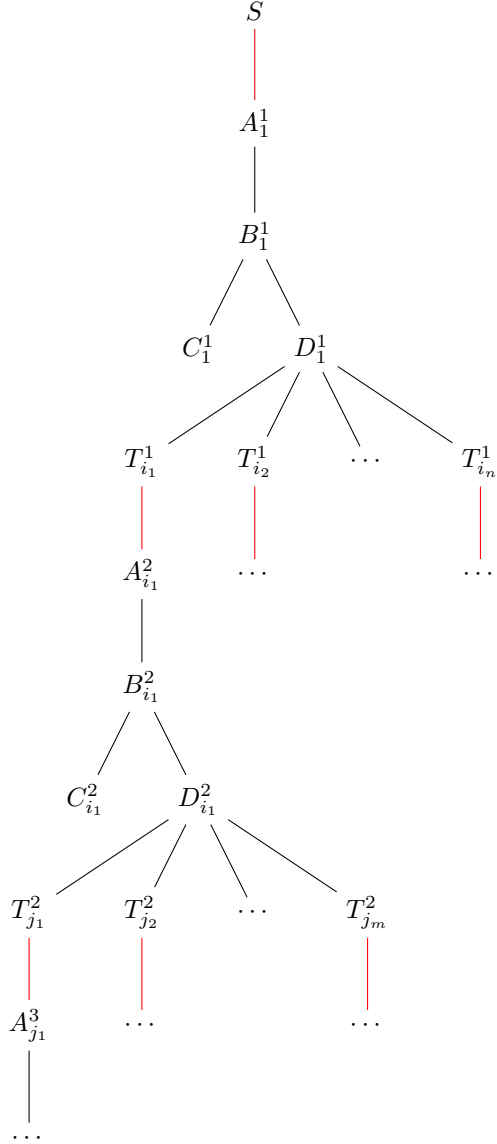

Suppose instead that $D^1_1\ne\emptyset$. For each $T^1_i\in D^1$, Proposition \ref{BHS19 5.11} implies $(\mathbf{F}_{T^1_i},\mathfrak{S}_{T^1_i})$ is an $F$-relative HHS. Applying the large link axiom to $(\mathbf{F}_{T^1_i},\mathfrak{S}_{T^1_i})$ with $W=T^1_i$, $\mathfrak{g}_{\mathbf{F}_{T^1_i}}(x)$, and $\mathfrak{g}_{\mathbf{F}_{T^1_i}}(y)$, we obtain a collection of domains $A^2_i=\{T^2_j\}\subseteq \mathfrak{S}_{T^1_i}-\{T^1_i\}$. Note that by \cite[Remark 1.16]{Behrstock_2017}, for any $T\sqsubseteq T^1_i$, we have $\pi_T\circ \mathfrak{g}_{\mathbf{F}_{T^1_i}}(x) = \pi_T(x)$. Therefore, for every $T\in\mathfrak{T}$ such that $T\sqsubsetneq T^1_i$ and  
$d_T(\pi_T(x),\pi_T(y))> F,$
we have that $d_T(\pi_T(\mathfrak{g}_{\mathbf{F}_{T^1_i}}(x)),\pi_T(\mathfrak{g}_{\mathbf{F}_{T^1_i}}(y)))> F,$
so there exists some $T^2_j\in A^2_i$ such that $T\sqsubseteq T^2_j$. Next, construct the sets $B^2_i,C^2_i,$ and $D^2_i$, analogously. If $D^2_i=\emptyset$, then this process can be terminated. If however there is some $T^2_j\in D^2_i$, then repeat, noting that $T^2_j\sqsubsetneq T^1_i$. Additionally, observe that the next step in this process will involve applying the large link axiom to the $F$-relative HHS $(\mathbf{F}_{T^2_j},\mathfrak{S}_{T^2_j})$ with $W=T^2_j$, $\mathfrak{g}_{\mathbf{F}_{T^2_j}}(x)$, and $\mathfrak{g}_{\mathbf{F}_{T^2_j}}(y)$. Repeating this process yields a chain 
$$T\sqsubseteq T^n_\ell \sqsubsetneq T^{n-1}_k \sqsubsetneq \cdots \sqsubsetneq T^2_j \sqsubsetneq T^1_i \sqsubsetneq S.$$
Note that we can vary the indices $i,j,k,$ etc. at each step, by choosing different elements of $D^p_q$, so we actually obtain many such chains. By finite complexity of $(\mathcal{X},\mathfrak{S})$, there can be no more than $E$ terms in any such chain, so $D^{E-1}_\ell$ must be empty for all indices $\ell$. Moreover, the large link axiom implies that each set $A^p_q$ is finite at every step, so there are finitely many chains obtained in this process. Finitely many chains of finite length then implies that the total number of domains produced in this process, given by
$$ \mathfrak{A} = \bigcup_{p}\bigcup_q A^p_q, $$
is finite. Let $N = |\mathfrak{A}|$. Recall that the sets $C^p_q\subseteq A^p_q$ consist of the domains provided by the large links axiom that are in $\mathfrak{T}$ and contain a domain $\mathfrak{T}$ that is relevant for $x$ and $y$. Additionally, let
$$ \mathfrak{C} = \bigcup_{p}\bigcup_q C^p_q.$$
Thus $\mathfrak{C}\subseteq\mathfrak{T}-\{S\}$ is a finite collection of domains such that for all $T\in\mathfrak{T}-\{S\}$, either $T\sqsubseteq U$ for some $U\in\mathfrak{C}$, or $\text{d}_T(\pi_T(x),\pi_T(y))\leq F\leq\hat{E}$, where $\hat{E}$ is as in \eqref{E hat large links}.

Finally, we will show that
$$ |\mathfrak{C}| \leq  \hat{E}\cdot d_{\mathcal{T}_S}(\pi_S(x),\pi_S(y))+\hat{E}. $$
Recall that $E'=\max(E,F)$. By the large links axiom,
$$ |\mathfrak{C}| \leq E'\cdot d_{CS}(\pi_{CS}(x),\pi_{CS}(y))+E' + \sum_{\substack{T_j\in \mathfrak{A}\\ T_j\notin\mathfrak{T}}} E'\cdot d_{T_j}(\pi_{T_j}\circ\mathfrak{g}_{\mathbf{F}_{T_j}}(x),\pi_{T_j}\circ\mathfrak{g}_{\mathbf{F}_{T_j}}(y))+E'.$$
Recall that for any domain $T_j\in\mathfrak{S}$, we have $\pi_{T_j}\circ\mathfrak{g}_{\mathbf{F}_{T_j}}(x)= \pi_{T_j}(x),$ so,
$$ |\mathfrak{C}| \leq E'\cdot d_{CS}(\pi_{CS}(x),\pi_{CS}(y))+E' + \sum_{\substack{T_j\in \mathfrak{A}\\ T_j\notin\mathfrak{T}}} E'\cdot d_{T_j}(\pi_{T_j}(x),\pi_{T_j}(y))+E'.$$
Then because $\mathfrak{A}$ has $N$ elements by definition,
$$ |\mathfrak{C}| \leq E'\cdot d_{CS}(\pi_{CS}(x),\pi_{CS}(y))+E' + NE'+\sum_{\substack{T_j\in A\\ T_j\notin\mathfrak{T}}} E'\cdot d_{T_j}(\pi_{T_j}(x),\pi_{T_j}(y)),$$
which can be rearranged to
$$ |\mathfrak{C}| \leq E'\cdot d_{CS}(\pi_{CS}(x),\pi_{CS}(y))+E' + NE'+E'\cdot\sum_{\substack{T_j\in \mathfrak{A}\\ T_j\notin\mathfrak{T}}}  d_{T_j}(\pi_{T_j}(x),\pi_{T_j}(y)).$$
Recall that $(\mathcal{T}_S,\mathfrak{S}-\mathfrak{S}^{M+})$ is an HHS where the associated $C(*)$, $\pi_*$, $\rho^*_*$, $\sqsubseteq$, $\perp$, $\pitchfork$ are the same as in $(\mathcal{X},\mathfrak{S})$. For any $T_j\in \mathfrak{A}$ such that $T_j\notin \mathfrak{T}$, because there exists $T\in\mathfrak{T}$ such that $T\sqsubsetneq T_j$, it must be the case that $T_j$ is not $\sqsubseteq$-minimal in $\mathfrak{S}$. Also $T_j$ cannot be contained in any $V\in\mathfrak{S}$ with $W\perp V$ such that diam$(CW)>M$, because this would imply $T_j\in\text{Unb}\subset\mathfrak{T}$. Therefore, $T_j\in\mathfrak{S}-\mathfrak{S}^{M+}$. The distance formula for $(\mathcal{T}_S,\mathfrak{S}-\mathfrak{S}^{M+})$ implies
\begin{align*}
\sum_{\substack{T_j\in \mathfrak{A}\\ T_j\notin\mathfrak{T}}} d_{T_j}(\pi_{T_j}(x),\pi_{T_j}(y)) &\leq Ns_0+\sum_{U\in\mathfrak{S}-\mathfrak{S}^{M+}} \{\!\{d_U(\pi_U(x),\pi_U(y))\}\!\}_{s_0} \\
&\leq K_0\cdot d_S(\pi_S(x),\pi_S(y))+C_0+Ns_0.
\end{align*}
Therefore,
\begin{align*}
|\mathfrak{C}| &\leq E'\cdot d_{CS}(\pi_{CS}(x),\pi_{CS}(y))+E' + NE'+E'\cdot\sum_{\substack{T_j\in \mathfrak{A}\\ T_j\notin\mathfrak{T}}}  d_{T_j}(\pi_{T_j}(x),\pi_{T_j}(y))\\
&\leq E'\cdot d_{CS}(\pi_{CS}(x),\pi_{CS}(y))+E' + NE'+E'(K_0\cdot d_S(\pi_S(x),\pi_S(y))+C_0+Ns_0)\\
&\leq E'\cdot d_S(\pi_S(x),\pi_S(y))+E' + NE'+E'(K_0\cdot d_S(\pi_S(x),\pi_S(y))+C_0+Ns_0)\\
&\leq E'(K_0+1)\cdot d_S(\pi_S(x),\pi_S(y))+ E'(1+ N(1+s_0)+C_0)\\
&\leq \hat{E}\cdot d_S(\pi_S(x),\pi_S(y))+ \hat{E},
\end{align*}
so this axiom is satisfied.

\noindent\\
\textbf{(11) Consistency:} If $V\pitchfork W$, then it remains true that
$$ \min\{\text{d}_W(\pi_W(x),\rho^V_W),
\text{d}_V(\pi_V(x),\rho^W_V)\}\leq E $$
for all $x\in\mathcal{X}$. Further, if $U\sqsubseteq V$ and $V \sqsubsetneq W$ or $V\pitchfork W$ and $W\perp U$, for $W\ne S$, then $\text{d}_W(\rho^U_W,\rho^V_W)\leq E$.

Thus, it remains to show that if $U\sqsubseteq V$ and $V \sqsubsetneq S$ then $\text{d}_S(\rho^U_S,\rho^V_S)\leq E$; note that this case was not discussed in \cite{abbott2021largest}. By definition, $\rho^V_S$ is the image of $\mathbf{F}_V$ under the factor map $\mathcal{X}\to\mathcal{T}_S$. Moreover, $U\sqsubsetneq V$ implies $\mathbf{F}_U\subset\mathbf{F}_V$, so $\rho^U_S\subset\rho^V_S$. Therefore, $\text{d}_S(\rho^U_S,\rho^V_S)=0\leq E$ and this condition is satisfied.

\noindent\\
\textbf{(12) Partial Realization:} Let $\{V_i\}$ be a finite collection of pairwise orthogonal elements of $\mathfrak{T}$ and $p_i\in CV_i$ for each $i$. First suppose $\{V_i\}=\{S\}$, and let $p\in\mathcal{T}_S$ be the chosen point. Because $\mathcal{T}_S$ is the cone-off of $\mathcal{X}$, there exists a point $x\in\mathcal{X}$ such that $d_S(\pi_S(x),p)\leq 1$, so the axiom is satisfied in this case.

Now suppose $\{V_i\}\ne\{S\}$. Let $x\in\mathcal{X}$ be the point provided by the partial realization axiom for $(\mathcal{X},\mathfrak{S})$. Then $d_{V_i}(\pi_{V_i}(x),p_i)\leq E$ still holds for all $i$ as the associated geodesic spaces in $(\mathcal{X},\mathfrak{S})$ and $(\mathcal{X},\mathfrak{T})$ are the same. Moreover, for all $i$ and for all domains $W\in\mathfrak{S}$ such that $W\ne S$ with $V_i\sqsubsetneq W$ or $W\pitchfork V_i$, then $d_W(\pi_W(x),\rho_W^{V_i})\leq E$ still holds.

It remains to show that $d_S(\pi_S(x),\rho_S^{V_i})\leq E'$ for any $i$; note that this case was not covered in \cite{abbott2021largest}. By partial realization for $(\mathcal{X},\mathfrak{S})$, 
$$d_{CS}(\pi_{CS}(x),\rho_{CS}^{V_i})\leq E.$$
Let $a=\mathfrak{g}_{V_i}(x)\in\mathbf{P}_{V_i}$, where $\mathfrak{g}_{V_i}$ is the gate map onto $\mathbf{P}_{V_i}$. Then by Proposition \ref{russell gates},
$$ d_{CS}(\pi_{CS}(a),\rho^{V_i}_{CS})\leq \mu.$$
The diameter of $\rho^{V_i}_{CS}$ is less than or equal to $E$, so
\begin{equation}
\label{partial realization}
d_{CS}(\pi_{CS}(x),\pi_{CS}(a))\leq 2E+\mu.  
\end{equation}
We will use the distance formula for the HHS $(\mathcal{T}_S,\mathfrak{S}-\mathfrak{S}^{M+})$ to bound $d_S(\pi_S(x),\pi_S(a))$. Let $K,C$ be the constants from Theorem \ref{distance formula} such that for all $w,z\in\mathcal{T}_S$,
$$ \frac{1}{K}\cdot d_S(w,z) - C \leq \sum_{U\in\mathfrak{S}-\mathfrak{S}^{M+}}\{\!\{ d_U(\pi_U(w),\pi_U(z)) \}\!\}_{s_0 + M + 2E+\mu},$$
where $s_0$ is the distance formula threshold.
Suppose towards contradiction that
$$ d_S(\pi_S(x),\pi_S(a)) > K+KC.$$
Thus
$$ 1 = \frac{1}{K}\cdot (K+KC)-C < \frac{1}{K}\cdot d_S(\pi_S(x),\pi_S(a)) - C \leq \sum_{U\in\mathfrak{S}-\mathfrak{S}^{M+}}\{\!\{ d_U(\pi_U(x),\pi_U(a)) \}\!\}_{s_0 + M + 2E+\mu}.$$
Therefore there exists some $U\in\mathfrak{S}-\mathfrak{S}^{M+}$ such that
\begin{equation}
\label{pr2}
d_U(\pi_U(x),\pi_U(a))\geq s_0+M+2E+\mu> 2E+\mu,  
\end{equation}
so \eqref{partial realization} implies $U\ne S$. Moreover, \eqref{pr2} and the bounded domain dichotomy property for $(\mathcal{X},\mathfrak{S})$ imply that $\mathbf{F}_U$ is unbounded. Since $V_i\sqsubsetneq S$ is an element of $\mathfrak{T}$, the domain $U$ cannot be orthogonal to $V_i$, lest $U\in \mathfrak{S}^{M+}$. Additionally $U$ cannot nest into $V_i$ for the following reasons. If $U=V_i\in(\mathfrak{T}-\{S\})\subset\mathfrak{S}^{M+}$, this is a contradiction. If $U\sqsubsetneq V_i$, then $V_i$ is not $\sqsubseteq$-minimal and thus $\mathbf{E}_{V_i}$ is unbounded, so $U\in\mathfrak{S}^{M+}$. Thus, either $U\pitchfork V_i$ or $V_i\sqsubsetneq U$. In either case, the partial realization axiom for $(\mathcal{X},\mathfrak{S})$ implies
$$ d_U(\pi_U(x),\rho^{V_i}_U)\leq E.$$
By Proposition \ref{russell gates}, because $a=\mathfrak{g}_{V_i}(x)$,
$$ d_U(\pi_U(a),\rho^{V_i}_U) = d_U(\pi_U\circ\mathfrak{g}_{V_i}(x),\rho^{V_i}_U)\leq \mu.$$
Finally, diam$(\rho^{V_i}_U)\leq E$ by definition, so
\begin{align*}
d_U(\pi_U(x),\pi_U(a))&\leq d_U(\pi_U(x),\rho^{V_i}_U)+\text{diam}(\rho^{V_i}_U)+d_U(\rho^{V_i}_U,\pi_U(a))\\
&\leq E + E +\mu\\
& = 2E+\mu,
\end{align*}
which contradicts \eqref{pr2}. Thus,
$$ d_S(\pi_S(x),\pi_S(a))\leq K+KC.$$
Moreover, since $a\in\mathbf{P}_{V_i}$, we have $\pi_S(a)\in\rho^{V_i}_S$. Thus,
$$ d_S(\pi_S(x),\rho^{V_i}_S)\leq K+KC.$$
Then taking $E'=K+KC+E$ completes this final axiom.
\end{proof}

\begin{definition}
We call the procedure in Theorem \ref{clean containers ABD} \emph{maximization} and say that the new structure $(\mathcal{X},\mathfrak{T})$ is the \emph{maximized structure}.
\end{definition}

\subsection{Bounded Projections if and only if Contracting}
\label{subsec bounded proj}

By performing the maximization procedure on a relative HHS, a new structure is constructed on which subspaces having bounded projections is equivalent to those subspaces being contracting. This equivalence is the basis for showing that the top level space in the maximized structure is a Morse detectability space. We begin by defining the notions of bounded projections and contracting.

\begin{definition}
Let $D>0$ and let $(\mathcal{X},\mathfrak{S})$ be a relative HHS. A subspace $\mathcal{Y}\subset \mathcal{X}$ has \emph{$D$-bounded projections} if diam$_U(\pi_U(\mathcal{Y}))<D$ for every $U\in\mathfrak{S}-\{S\}$.
\end{definition}

\begin{definition}
\label{contracting}
A subspace $\mathcal{Y}$ in a metric space $\mathcal{X}$ is \emph{$D$-contracting} if there exists a map $\pi_{\mathcal{Y}}\colon X\to\mathcal{Y}\subset X$ and constants $D>0$ and $A>1$ satisfying:
\begin{enumerate}
    \item for any $x\in\mathcal{Y}$, we have $d_X(x,\pi_\mathcal{Y}(x))<D$;
    \item if $x,y\in X$ with $d_X(x,y)<1$, then $d_X(\pi_\mathcal{Y}(x),\pi_\mathcal{Y}(y))<D$; and
    \item for all $x\in X$, if we set $\displaystyle{R=\frac{1}{A}\cdot d_X(x,\mathcal{Y})}$, then diam$_X(\pi_\mathcal{Y}(B_R(x)))\leq D$.
\end{enumerate}  
\end{definition}

The following theorem generalizes \cite[Theorem 4.4]{abbott2021largest} and illustrates the conditions for which being contracting and having bounded projections are equivalent. While this proof follows similar lines as in \cite{abbott2021largest}, we take this opportunity to fill in missing details.

\begin{thm}
\label{abd4.4}
Let $\mathcal{X}$ be a geodesic metric space and let $(\mathcal{X},\mathfrak{S})$ be an $E$-relative HHS with $|\mathfrak{S}|>1$. For any $D>0$ and $K\geq 1$ there exists $D'>0$ depending only on $D$ and $(\mathcal{X},\mathfrak{S})$ such that the following holds for every $(K,K)$-quasi-isometric embedding $\gamma\colon\mathcal{Y}\to\mathcal{X}$. If $\gamma(\mathcal{Y})$ has $D$-bounded projections, then $\gamma(\mathcal{Y})$ is $D'$-contracting. Moreover, if $(\mathcal{X},\mathfrak{S})$ has the bounded domain dichotomy, clean containers, and unbounded minimal products, then $\mathcal{X}$ admits a relatively hierarchically hyperbolic structure $(\mathcal{X},\mathfrak{T})$ with unbounded products where, additionally, if $\gamma$ is an $(M;K,K)$-stable embedding, then $\gamma(\mathcal{Y})$ has $D'$-bounded projections.
\end{thm}

\begin{proof} Fix $D>0$ and $K\geq 1$. Let $\gamma\colon\mathcal{Y}\to \mathcal{X}$ be a $(K,K)$-quasi-isometric embedding, and let $\gamma(\mathcal{Y})$ have $D$-bounded projections. 

\begin{claim}
    The set $\gamma(\mathcal{Y})$ is a hierarchically quasi-convex subset of $\mathcal{X}$
\end{claim}

\begin{claimproof}
Because $\gamma(\mathcal{Y})$ has $D$-bounded projections, the first two conditions of Definition \ref{hqconvex} are clearly satisfied for all $U\in\mathfrak{S}-\{S\}$. We now show $\pi_S\circ\gamma(\mathcal{Y})$ is quasi-convex. Fix two points $\pi_S(a),\pi_S(b)\in\pi_S\circ\gamma(\mathcal{Y})$. The distance formula implies there exists $k',c'$ such that
\begin{equation}
\label{distform4.4}
d_{\mathcal{X}}(a,b)\asymp_{k',c'} \sum_{U\in\mathfrak{S}}\{\!\{ d_U(\pi_U(a),\pi_U(b)) \}\!\}_{s_0+D} = d_S(\pi_S(a),\pi_S(b)), 
\end{equation}
where $s_0$ is the distance formula threshold for $(\mathcal{X},\mathfrak{S})$ from Theorem \ref{distance formula}. Since $\gamma$ is a quasi-isometric embedding, the composition $\pi_S\circ\gamma$ is a $(K',K')$-quasi-isometric embedding of $\mathcal{Y}$ into the hyperbolic (because $|\mathfrak{S}|>1$) space $CS$, for some constant $K'$. Thus $\pi_S\circ\gamma$ is a stable embedding by Lemma \ref{qi hyp stable} and therefore quasi-convex. For condition (3), fix some $x\in\mathcal{X}$ and $R\geq 0$ such that $d_U(\pi_U(x),\pi_U\circ\gamma(\mathcal{Y}))\leq R$ for every $U\in\mathfrak{S}$. Thus there exists a point $\pi_S(x')\in\pi_S\circ\gamma(\mathcal{Y})$ such that $d_S(\pi_S(x),\pi_S(x'))\leq R+1$. By \eqref{distform4.4}, 
$$d_{\mathcal{X}}(x,\gamma(\mathcal{Y}))\leq d_{\mathcal{X}}(x,x')\leq k'\cdot d_S(\pi_S(x),\pi_S(x'))+c'\leq k'(R+1)+c',$$
so condition (3) is satisfied.  
\end{claimproof}

Because $\gamma(\mathcal{Y})$ is hierarchically quasi-convex, Proposition \ref{general gates} implies that there exist a constant $\mu\geq 1$ and a $(\mu,\mu)$-coarsely Lipschitz gate map $\mathfrak{g}_{\mathcal{Y}}\colon\mathcal{X}\to\gamma(\mathcal{Y})$. The space $CS$ is $E$-hyperbolic, so Proposition \ref{general gates} states $d_S(\pi_S\circ\mathfrak{g}_{\mathcal{Y}}(x),p_S\circ\pi_S(x))\leq \mu$ for all $x\in\mathcal{X}$, where $p_S$ is the (coarse) projection $CS\to\pi_S\circ\gamma(\mathcal{Y})$. We will show that $\mathfrak{g}_{\mathcal{Y}}$ satisfies the three conditions from Definition \ref{contracting} for constant $D'$ determined throughout the course of the proof.

For condition (1), first fix some $x\in\gamma(\mathcal{Y})$. Then $\pi_S(x)\in\pi_S\circ\gamma(\mathcal{Y})$. By \cite[Lemma 11.53]{dructu2018geometric}, $d_S(\pi_S(x),p_S\circ\pi_S(x))\leq \Delta$, where $\Delta$ depends on $E$ and the quasi-convexity constant of $\gamma(\mathcal{Y})$. Define the constant 
$$L=s_0+D+E+\mu+2.$$
Because $\gamma(\mathcal{Y})$ has $D$-bounded projections, by taking the distance formula threshold to be $L$, there exists $K_1$ and $C_1$ depending on 
$(\mathcal{X},\mathfrak{S})$ such that
\begin{align*}
    d_{\mathcal{X}}(x,\mathfrak{g}_{\mathcal{Y}}(x)) &\leq K_1\cdot d_S(\pi_S(x),\pi_S\circ\mathfrak{g}_{\mathcal{Y}}(x))+C_1\\
    &\leq K_1(d_S(\pi_S(x),p_S\circ\pi_S(x))+d_S(p_S\circ\pi_S(x),\pi_S\circ\mathfrak{g}_{\mathcal{Y}}(x)))+C_1\\
    &\leq K_1(\Delta+\mu)+C_1.
\end{align*}
Thus condition (1) is satisfied for the constant $D_1= K_1(\Delta+\mu)+C_1$.

Condition (2) is satisfied for the constant $D_2=2\mu$, as $\mathfrak{g}_{\mathcal{Y}}$ is $(\mu,\mu)$-coarsely Lipschitz: for any $x,y\in\mathcal{X}$ such that $d_{\mathcal{X}}(x,y)<1$,
$$ d_{\mathcal{X}}(\mathfrak{g}_{\mathcal{Y}}(x),\mathfrak{g}_{\mathcal{Y}}(y))\leq \mu\cdot d_{\mathcal{X}}(x,y)+\mu< 2\mu. $$

Finally, we will verify condition (3). Fix some $x\in\mathcal{X}$, let $A=2K_1(C_1+K_1)$, and fix any point $y\in\mathcal{X}$ such that 
$$d_{\mathcal{X}}(x,y)< \frac{1}{A}d_{\mathcal{X}}(x,\gamma(\mathcal{Y})).$$
If $d_{\mathcal{X}}(x,y)<1$, then we are done by condition (2). Thus, suppose $d_{\mathcal{X}}(x,y)\geq 1$. Because $\gamma(\mathcal{Y})$ has $D$-bounded projections, for any $U\in\mathfrak{S}-\{S\}$, the distance formula implies
$$
d_{\mathcal{X}}(\mathfrak{g}_{\mathcal{Y}}(x),\mathfrak{g}_{\mathcal{Y}}(y))\leq K_1\cdot d_S(\pi_S\circ\mathfrak{g}_{\mathcal{Y}}(x),\pi_S\circ\mathfrak{g}_{\mathcal{Y}}(y))+C_1.
$$
Then the triangle inequality yields
\begin{equation}
\label{4.4Sbound}
d_{\mathcal{X}}(\mathfrak{g}_{\mathcal{Y}}(x),\mathfrak{g}_{\mathcal{Y}}(y))\leq K_1\cdot d_S(p_S\circ\pi_S(x),p_S\circ\pi_S(y))+2\mu K_1+C_1.
\end{equation}
Therefore, it suffices to bound the distance between the nearest point projections of $\pi_S(x)$ and $\pi_S(y)$ onto $\pi_S\circ\gamma(\mathcal{Y})$. 

Let $K_2,C_2$ be the distance formula constants for a threshold of $2L$. Then using the distance formula, we obtain
\begin{align*}
\sum_{U\in\mathfrak{S}}\{\!\{ d_U(\pi_U(x),\pi_U\circ\mathfrak{g}_{\mathcal{Y}}(x)) \}\!\}_{2L} &\geq \frac{1}{K_2}d_{\mathcal{X}}(x,\mathfrak{g}_{\mathcal{Y}}(x))-C_2\\
&\geq \frac{1}{K_2}d_{\mathcal{X}}(x,\gamma(\mathcal{Y}))-C_2\\
&>\frac{A}{K_2}d_{\mathcal{X}}(x,y)-C_2\\
&=\frac{K_2(C_2+C_1+2K_1)}{K_2}d_{\mathcal{X}}(x,y)-C_2\\
&=(C_2+C_1+2K_1)\cdot d_{\mathcal{X}}(x,y)-C_2\\
&=2K_1\cdot d_{\mathcal{X}}(x,y)+(C_2+C_1)\cdot d_{\mathcal{X}}(x,y)-C_2\\
&\geq 2K_1\cdot d_{\mathcal{X}}(x,y)+(C_2+C_1)-C_2\\
&= 2K_1\cdot d_{\mathcal{X}}(x,y)+C_1\\
&\geq \frac{2K_1}{K_1}\sum_{U\in\mathfrak{S}}\{\!\{ d_U(\pi_U(x),\pi_U(y)) \}\!\}_L - C_1+C_1\\
&= 2\sum_{U\in\mathfrak{S}}\{\!\{ d_U(\pi_U(x),\pi_U(y)) \}\!\}_L\\
&\geq \sum_{U\in\mathfrak{S}}\{\!\{ d_U(\pi_U(x),\pi_U(y)) +L\}\!\}_{2L}.
\end{align*}
Therefore there exists some $W\in\mathfrak{S}$ such that
\begin{equation}
\label{double inequality}
d_W(\pi_W(x),\pi_W\circ\mathfrak{g}_{\mathcal{Y}}(x))\geq d_W(\pi_W(x),\pi_W(y))+L.
\end{equation}

First suppose $W=S$. The space $CS$ is $E$-hyperbolic. By the triangle inequality and \eqref{double inequality}, we have
\begin{align*}
d_S(\pi_S(x),\pi_S\circ\gamma(\mathcal{Y}))&\geq d_S(\pi_S(x),p_S\circ\pi_S(x))\\
&\geq d_S(\pi_S(x),\pi_S\circ\mathfrak{g}_{\mathcal{Y}}(x))-d_S(\pi_S\circ\mathfrak{g}_{\mathcal{Y}}(x),p_S\circ\pi_S(x))\\
&\geq  d_S(\pi_S(x),\pi_S\circ\mathfrak{g}_{\mathcal{Y}}(x))-\mu\\
&\geq  d_S(\pi_S(x),\pi_S(y))+L-\mu\\
&>  d_S(\pi_S(x),\pi_S(y))+1.
\end{align*}
Thus any geodesic from $\pi_S(x)$ to $\pi_S(y)$ is disjoint from $\pi_S\circ\gamma(\mathcal{Y})$. Because $\pi_S\circ\gamma(\mathcal{Y})$ is a quasi-isometric embedding in an $E$-hyperbolic space, it is a well known fact in hyperbolic geometry that $d_S(p_S\circ\pi_S(x),p_S\circ\pi_S(y))$ is bounded by a uniform constant depending only on $E$, which we call $Q_1$. It follows from \eqref{4.4Sbound} that condition (3) is satisfied in this case for the constant $D_3^1=K_1Q_1+2\mu K_1+C_1$.

Now suppose instead that $W\ne S$ and

\tikzset{every picture/.style={line width=0.75pt}} 

\begin{figure}
\centering
\begin{tikzpicture}[x=0.75pt,y=0.75pt,yscale=-1,xscale=1]

\draw    (151,378) .. controls (191,348) and (399,405.6) .. (439,375.6) ;
\draw  [dash pattern={on 0.84pt off 2.51pt}]  (144,108.6) .. controls (201,142.6) and (206,212.6) .. (206,237.6) ;
\draw [shift={(206,237.6)}, rotate = 90] [color={rgb, 255:red, 0; green, 0; blue, 0 }  ][fill={rgb, 255:red, 0; green, 0; blue, 0 }  ][line width=0.75]      (0, 0) circle [x radius= 3.35, y radius= 3.35]   ;
\draw [shift={(144,108.6)}, rotate = 30.82] [color={rgb, 255:red, 0; green, 0; blue, 0 }  ][fill={rgb, 255:red, 0; green, 0; blue, 0 }  ][line width=0.75]      (0, 0) circle [x radius= 3.35, y radius= 3.35]   ;
\draw  [dash pattern={on 0.84pt off 2.51pt}]  (414,126.6) .. controls (362,147.6) and (372,238.6) .. (373,238.6) ;
\draw [shift={(373,238.6)}, rotate = 0] [color={rgb, 255:red, 0; green, 0; blue, 0 }  ][fill={rgb, 255:red, 0; green, 0; blue, 0 }  ][line width=0.75]      (0, 0) circle [x radius= 3.35, y radius= 3.35]   ;
\draw [shift={(414,126.6)}, rotate = 158.01] [color={rgb, 255:red, 0; green, 0; blue, 0 }  ][fill={rgb, 255:red, 0; green, 0; blue, 0 }  ][line width=0.75]      (0, 0) circle [x radius= 3.35, y radius= 3.35]   ;
\draw   (256,237.6) .. controls (256,221.42) and (269.12,208.3) .. (285.3,208.3) .. controls (301.48,208.3) and (314.6,221.42) .. (314.6,237.6) .. controls (314.6,253.78) and (301.48,266.9) .. (285.3,266.9) .. controls (269.12,266.9) and (256,253.78) .. (256,237.6) -- cycle ;
\draw  [dash pattern={on 0.84pt off 2.51pt}]  (206,237.6) .. controls (203,297.6) and (194,325.6) .. (151,378) ;
\draw [shift={(151,378)}, rotate = 129.37] [color={rgb, 255:red, 0; green, 0; blue, 0 }  ][fill={rgb, 255:red, 0; green, 0; blue, 0 }  ][line width=0.75]      (0, 0) circle [x radius= 3.35, y radius= 3.35]   ;
\draw [shift={(206,237.6)}, rotate = 92.86] [color={rgb, 255:red, 0; green, 0; blue, 0 }  ][fill={rgb, 255:red, 0; green, 0; blue, 0 }  ][line width=0.75]      (0, 0) circle [x radius= 3.35, y radius= 3.35]   ;
\draw  [dash pattern={on 0.84pt off 2.51pt}]  (373,238.6) .. controls (374,238.6) and (366,321.6) .. (439,375.6) ;
\draw [shift={(439,375.6)}, rotate = 36.49] [color={rgb, 255:red, 0; green, 0; blue, 0 }  ][fill={rgb, 255:red, 0; green, 0; blue, 0 }  ][line width=0.75]      (0, 0) circle [x radius= 3.35, y radius= 3.35]   ;
\draw [shift={(373,238.6)}, rotate = 0] [color={rgb, 255:red, 0; green, 0; blue, 0 }  ][fill={rgb, 255:red, 0; green, 0; blue, 0 }  ][line width=0.75]      (0, 0) circle [x radius= 3.35, y radius= 3.35]   ;
\draw  [dash pattern={on 0.84pt off 2.51pt}]  (206,237.6) -- (256,237.6) ;
\draw  [dash pattern={on 0.84pt off 2.51pt}]  (314.6,237.6) -- (373,238.6) ;

\draw (324,212) node [anchor=north west][inner sep=0.75pt]    {$\leq E$};
\draw (214,212) node [anchor=north west][inner sep=0.75pt]    {$\leq E$};
\draw (270,224) node [anchor=north west][inner sep=0.75pt]    {$\rho _{S}^{W}$};
\draw (268,381) node [anchor=north west][inner sep=0.75pt]    {$\pi _{S}\circ \gamma(\mathcal{Y})$};
\draw (420,389) node [anchor=north west][inner sep=0.75pt]    {$p_{S} \circ \pi _{S}( y)$};
\draw (124,384) node [anchor=north west][inner sep=0.75pt]    {$p_{S} \circ \pi _{S}( x)$};
\draw (387,227) node [anchor=north west][inner sep=0.75pt]    {$z$};
\draw (174,227) node [anchor=north west][inner sep=0.75pt]    {$w$};
\draw (423,94) node [anchor=north west][inner sep=0.75pt]    {$\pi _{S}( y)$};
\draw (121,78) node [anchor=north west][inner sep=0.75pt]    {$\pi _{S}( x)$};

\end{tikzpicture}
\caption{A depiction of $CS$ in the proof of condition (3) that $\gamma(\mathcal{Y})$ is $D'$-contracting.}
\label{4.4Picture}
\end{figure}
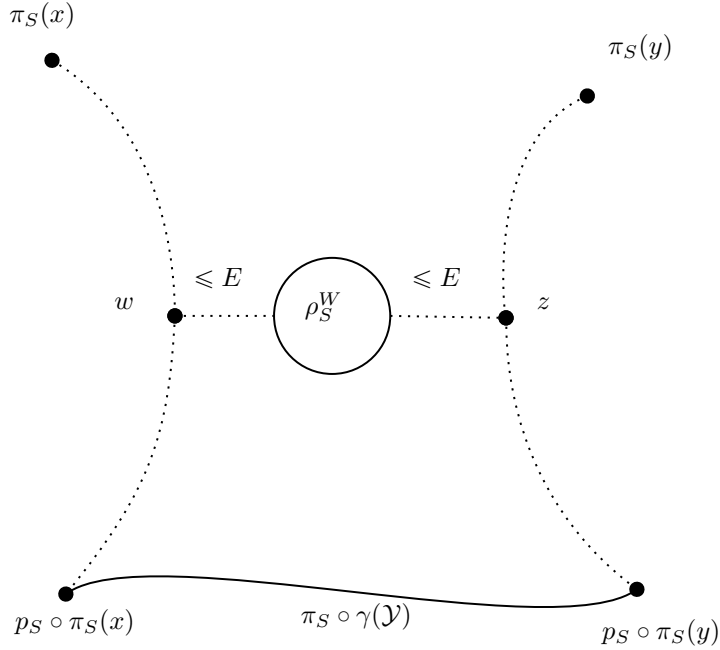

$$d_S(\pi_S(x),p_S\circ\pi_S(x))<d_S(\pi_S(x),\pi_S(y))+L.$$ 
Since $\gamma(\mathcal{Y})$ has $D$-bounded projections, it follows from the triangle inequality that
\begin{align*}
d_W(\pi_W(y),\pi_W\circ\mathfrak{g}_{\mathcal{Y}}(y))\geq& d_W(\pi_W(x),\pi_W\circ\mathfrak{g}_{\mathcal{Y}}(x))\\
&- d_W(\pi_W(x),\pi_W(y))\\
&-d_W(\pi_W\circ\mathfrak{g}_{\mathcal{Y}}(x),\pi_W\circ\mathfrak{g}_{\mathcal{Y}}(y))\\
\geq & d_W(\pi_W(x),\pi_W(y)) + L - d_W(\pi_W(x),\pi_W(y))-D\\
=& L-D\\
\geq& E+\mu.
\end{align*}
Additionally, by \eqref{double inequality},
$$ d_W(\pi_W(x),\pi_W\circ\mathfrak{g}_{\mathcal{Y}}(x))\geq d_W(\pi_W(x),\pi_W(y))+L\geq L\geq E+\mu.$$
Because both $d_W(\pi_W(x),\pi_W\circ\mathfrak{g}_{\mathcal{Y}}(x))\geq E+\mu$ and $d_W(\pi_W(y),\pi_W\circ\mathfrak{g}_{\mathcal{Y}}(y))\geq E+\mu$, the bounded geodesic image axiom for $(\mathcal{X},\mathfrak{S})$ implies that any geodesic between $\pi_S(x)$ and $p_S\circ\pi_S(x)$ (as well as between $\pi_S(y)$ and $p_S\circ\pi_S(y)$) must intersect the $E$-neighborhood of $\rho^W_S$. Let
$$ w\in [\pi_S(x),p_S\circ\pi_S(x)]\cap\mathcal{N}_E(\rho^W_S) \qquad\text{and}\qquad z\in [\pi_S(y),p_S\circ\pi_S(y)]\cap\mathcal{N}_E(\rho^W_S),$$
as seen in Figure \ref{4.4Picture}. Thus $d_S(w,z)\leq 3E$.  By definition, $p_S\circ\pi_S(x)=p_S(w)$ and $p_S\circ\pi_S(y)=p_S(z)$. If $d_S(w,p_S\circ\pi_S(x))>3E$, then again $d_S(p_S\circ\pi_S(x),p_S\circ\pi_S(y))$ is bounded by some uniform constant depending only on $E$, which we call $Q_2$. It follows from \eqref{4.4Sbound} that condition (3) is satisfied in this case for the constant $D^2_3=K_1Q_2+2\mu K_1+C_1$.

Suppose instead that $d_S(w,p_S\circ\pi_S(x))\leq 3E$. Since $p_S$ is a nearest point projection,
$$d_S(z,p_S\circ\pi_S(y)) \leq d_S(z,w)+d_S(w,p_S\circ\pi_S(x))\leq 6E.$$
Therefore, 
$$ d_S(p_S\circ\pi_S(x),p_S\circ\pi_S(y))\leq d_S(p_S\circ\pi_S(x),w)+d_S(w,z)+d_S(z,p_S\circ\pi_S(y))\leq 12E.$$
Using \eqref{4.4Sbound} completes the proof of the final case of condition (3) for the constant $D^3_3=(12E+2\mu)K_1+C_1$. By taking $D'=\max\{D_1,D_2,D_3^1,D_3^2,D_3^3\}$, we have shown that if $\gamma(\mathcal{Y})$ has $D$-bounded projections, then $\gamma(\mathcal{Y})$ is $D'$-contracting.

\begin{figure}[h!]
    \centering

\begin{tikzpicture}[x=0.75pt,y=0.75pt,yscale=-1,xscale=1]

\draw    (93,325.6) -- (474,325.6) ;
\draw    (93,325.6) -- (93,392.6) ;
\draw    (474,325.6) -- (474,397.6) ;
\draw    (94,180.6) .. controls (118,111.6) and (120,207.6) .. (171,163.6) ;
\draw [shift={(94,180.6)}, rotate = 289.18] [color={rgb, 255:red, 0; green, 0; blue, 0 }  ][fill={rgb, 255:red, 0; green, 0; blue, 0 }  ][line width=0.75]      (0, 0) circle [x radius= 3.35, y radius= 3.35]   ;
\draw    (171,163.6) .. controls (199,106.6) and (325,209.6) .. (366,160.6) ;
\draw [shift={(171,163.6)}, rotate = 296.16] [color={rgb, 255:red, 0; green, 0; blue, 0 }  ][fill={rgb, 255:red, 0; green, 0; blue, 0 }  ][line width=0.75]      (0, 0) circle [x radius= 3.35, y radius= 3.35]   ;
\draw    (366,160.6) .. controls (435,140.6) and (457,182.6) .. (483,175.6) ;
\draw [shift={(483,175.6)}, rotate = 344.93] [color={rgb, 255:red, 0; green, 0; blue, 0 }  ][fill={rgb, 255:red, 0; green, 0; blue, 0 }  ][line width=0.75]      (0, 0) circle [x radius= 3.35, y radius= 3.35]   ;
\draw [shift={(366,160.6)}, rotate = 343.84] [color={rgb, 255:red, 0; green, 0; blue, 0 }  ][fill={rgb, 255:red, 0; green, 0; blue, 0 }  ][line width=0.75]      (0, 0) circle [x radius= 3.35, y radius= 3.35]   ;
\draw [color={rgb, 255:red, 208; green, 2; blue, 27 }  ,draw opacity=1 ]   (166,244.6) .. controls (209,289.6) and (382,278.6) .. (400,251.6) ;
\draw    (94,180.6) .. controls (100,235.6) and (150,228.6) .. (166,244.6) ;
\draw [shift={(166,244.6)}, rotate = 45] [color={rgb, 255:red, 0; green, 0; blue, 0 }  ][fill={rgb, 255:red, 0; green, 0; blue, 0 }  ][line width=0.75]      (0, 0) circle [x radius= 3.35, y radius= 3.35]   ;
\draw    (400,251.6) .. controls (454,189.6) and (508,272.6) .. (483,175.6) ;
\draw [shift={(400,251.6)}, rotate = 311.05] [color={rgb, 255:red, 0; green, 0; blue, 0 }  ][fill={rgb, 255:red, 0; green, 0; blue, 0 }  ][line width=0.75]      (0, 0) circle [x radius= 3.35, y radius= 3.35]   ;
\draw  [dash pattern={on 4.5pt off 4.5pt}]  (93,325.6) .. controls (88,225.6) and (273,225.6) .. (302,226.6) .. controls (331,227.6) and (438,240.6) .. (474,325.6) ;
\draw  [dash pattern={on 0.84pt off 2.51pt}]  (295,227.6) -- (310,326.6) ;
\draw [color={rgb, 255:red, 90; green, 74; blue, 226 }  ,draw opacity=1 ]   (171,163.6) -- (166,244.6) ;
\draw [color={rgb, 255:red, 90; green, 74; blue, 226 }  ,draw opacity=1 ]   (366,160.6) -- (400,251.6) ;

\draw (298,347) node [anchor=north west][inner sep=0.75pt]    {$\mathbf{P}_{U}{}_{_{i}}$};
\draw (75,165) node [anchor=north west][inner sep=0.75pt]    {$x_{i}$};
\draw (490,160) node [anchor=north west][inner sep=0.75pt]    {$y_{i}$};
\draw (154,145) node [anchor=north west][inner sep=0.75pt]    {$w_{i}$};
\draw (358,140) node [anchor=north west][inner sep=0.75pt]    {$z_{i}$};
\draw (155,255) node [anchor=north west][inner sep=0.75pt]    {$a_{i}$};
\draw (396,259) node [anchor=north west][inner sep=0.75pt]    {$b_{i}$};
\draw (348,275) node [anchor=north west][inner sep=0.75pt]    {$\beta _{i}$};
\draw (480,230) node [anchor=north west][inner sep=0.75pt]    {$\alpha _{i}$};
\draw (433,145) node [anchor=north west][inner sep=0.75pt]    {$\gamma _{i}$};
\draw (143,196) node [anchor=north west][inner sep=0.75pt]    {$\zeta _{i}$};
\draw (358,194) node [anchor=north west][inner sep=0.75pt]    {$\xi _{i}$};
\draw (288,300) node [anchor=north west][inner sep=0.75pt]    {$\nu _{\lambda }$};
\draw (175,194) node [anchor=north west][inner sep=0.75pt]    {$\leq M'( \lambda ,\lambda )$};
\draw (384,193.1) node [anchor=north west][inner sep=0.75pt]    {$\leq M'( \lambda ,\lambda )$};

\end{tikzpicture}

    \caption{A depiction of $\mathcal{X}$ in the proof that $\gamma(\mathcal{Y})$ has $D'$-bounded projections.}
    \label{4.4imagept2}
\end{figure}
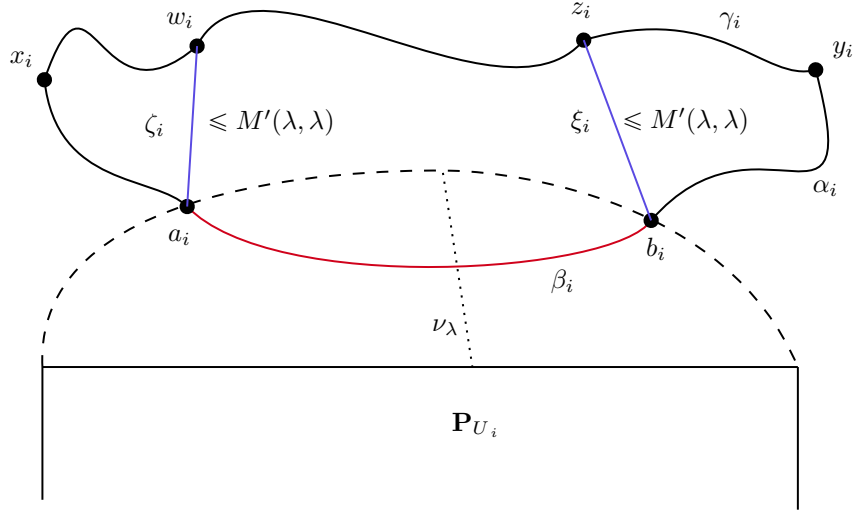

\medskip

We are now ready to prove the ``moreover" statement of the theorem. Let $(\mathcal{X},\mathfrak{S})$ be an $E$-relatively hierarchically hyperbolic space with $|\mathfrak{S}|>1$, the bounded domain dichotomy, clean containers, and unbounded minimal products. By Theorem \ref{clean containers ABD}, we obtain a new structure $(\mathcal{X},\mathfrak{T})$ which has relatively unbounded products, where $\mathfrak{T}=\{S\}\sqcup\text{Unb}\sqcup\text{Min}$. Because $(\mathcal{X},\mathfrak{S})$ has unbounded minimal products, Min is empty, so in fact $(\mathcal{X},\mathfrak{T})$ has unbounded products. 

Suppose towards contradiction that $\gamma(\mathcal{Y})$ does not have $D'$-bounded projections for any $D'$. Then there exists a sequence of domains $U_i\in\mathfrak{T}-\{S\}$ such that diam$(\pi_{U_i}\circ\gamma(\mathcal{Y}))\to\infty$ as $i\to\infty$. Choose a sequence of pairs of points $x_i,y_i\in\gamma(\mathcal{Y})$ such that $d_{U_i}(\pi_{U_i}(x_i),\pi_{U_i}(y_i))=K_i$, where $K_i\to\infty$ as $i\to\infty$. Let $\lambda$ be sufficiently large such that Propositions \ref{active subpaths} and \ref{hpaths} hold. Then there exists a $(\lambda,\lambda)$-hierarchy path $\alpha_i$ connecting each pair $x_i,y_i$. The path $\alpha_i$ is a $(\lambda,\lambda)$-quasi-geodesic with endpoints on the image of the $(M;K,K)$-stable embedding $\gamma$, so it is contained in the $M'(\lambda,\lambda)$-neighborhood of $\gamma(\mathcal{Y})$ by Lemma \ref{stable is morse}, where $M'$ depends only on $M$ and $K$.

By Proposition \ref{active subpaths}, there exists a subpath $\beta_i\subset\alpha_i$ such that $\beta_i\subset\mathcal{N}_{\nu_{\lambda}}(\mathbf{P}_{U_i})$. Moreover, by the third bullet point in Proposition \ref{active subpaths},
$$\text{diam}_{U_i}(\pi_{U_i}(\beta_i))\geq d_{U_i}(\pi_{U_i}(x_i),\pi_{U_i}(y_i))-24(\lambda E+E)\geq K_i-24(\lambda E+E).$$
Additionally, since the projection maps $\pi_{U_i}$ are $(E,E)$-coarsely Lipshitz, we have
$$ \text{diam}_{\mathcal{X}}(\beta_i)\geq \frac{1}{E}\text{diam}_{U_i}(\pi_{U_i}(\beta_i))-E\geq \frac{K_i}{E}-24(\lambda+1)-E.$$
Thus, $\text{diam}_{\mathcal{X}}(\beta_i)\to\infty$ as $i\to\infty$, and hence there exist points $a_i,b_i\in\beta_i$ such that $d_{\mathcal{X}}(a_i,b_i)\to\infty$ as $i\to\infty$. Let $w_i,z_i\in\gamma(\mathcal{Y})$ be points such that $d_{\mathcal{X}}(a_i,w_i)\leq M'(\lambda,\lambda)$ and $d_{\mathcal{X}}(b_i,z_i)\leq M'(\lambda,\lambda)$. By the triangle inequality,
$$ d_{\mathcal{X}}(w_i,z_i) \geq d_{\mathcal{X}}(a_i,b_i)-2M'(\lambda,\lambda),$$
so $d_{\mathcal{X}}(w_i,z_i)\to\infty$ as $i\to\infty$. Let $\zeta_i=[w_i,a_i]$ and $\xi_i=[z_i,b_i]$ be geodesics in $\mathcal{X}$. The concatenated path $\zeta_i\ast\beta_i|_{[a_i,b_i]}\ast\xi_i$ is a $(\lambda,\lambda+2M'(\lambda,\lambda))$-quasi-geodesic with endpoints in $\gamma(\mathcal{Y})$. Lemma \ref{stable is morse} implies there is an $(M';K,K)$-Morse quasi-geodesic $\eta_i\subset\gamma(\mathcal{Y})$ contained in the $M'(\lambda,\lambda+2M'(\lambda,\lambda))$-neighborhood of $\zeta_i\ast\beta_i|_{[a_i,b_i]}\ast\xi_i$ and therefore in the $(M'(\lambda,\lambda+2M'(\lambda,\lambda))+M'(\lambda,\lambda)+\nu_{\lambda})$-neighborhood of $\mathbf{P}_{U_i}$. 

Because $d_{\mathcal{X}}(w_i,z_i)\to\infty$ as $i\to\infty$, we have shown that the arbitrarily long $(M';K,K)$-Morse quasi-geodesics $\eta_i$ are uniformly close to a direct product with unbounded factors. Such a direct product is uniformly Morse limited by \cite[Theorem A.3]{drutu2025weakmorsepropertiesspaces}, and so this is a contradiction, proving the ``moreover" statement and concluding the proof of the theorem.
\end{proof}

\section{Stability in a Relative HHS}
\label{gp}

The main goal of this section is to show that in many cases, the top level space $CS$ associated to the maximized relative HHS structure produced by Theorem \ref{clean containers ABD} is a Morse detectability space. To prove that $CS$ is a Morse detectability space, we must show that quasi-isometric embeddings in $\mathcal{X}$ project to quasi-isometric embeddings in $CS$ if and only if they are stable embeddings. 

The next lemma shows that projecting to a quasi-isometric embedding in $CS$ and having bounded projections are equivalent in a relative HHS. Note that there are no additional assumptions on the relative HHS such as the bounded domain dichotomy, clean containers, or unbounded products.

\begin{lemma}
\label{bounded proj iff CS qgeo}
Let $(\mathcal{X},\mathfrak{S})$ be an relative HHS and let $\gamma\colon\mathcal{Y}\to\mathcal{X}$ be a quasi-isometric embedding. The projection  $\pi_S\circ\gamma$ is a quasi-isometric embedding into $CS$ if and only if $\gamma(\mathcal{Y})$ has bounded projections.
\end{lemma}

\begin{proof}
Let $(\mathcal{X},\mathfrak{S})$ be an $E$-relative HHS. The statement is vacuously true if $|\mathfrak{S}|=1$, so suppose $|\mathfrak{S}|>1$, and in particular, $CS$ is hyperbolic. Let $\gamma\colon\mathcal{Y}\to\mathcal{X}$ be a $(\lambda,\varepsilon)$-quasi-isometric embedding. For the first direction, suppose $\gamma(\mathcal{Y})$ has $D$-bounded projections. If $s_0$ is the minimum distance formula threshold for $(\mathcal{X},\mathfrak{S})$ from Theorem \ref{distance formula}, then there exist constants $K,C$ such that for any $x,y\in\gamma$
$$ d_{\mathcal{X}}(x,y) \asymp_{K,C} \sum_{U\in\mathfrak{S}}\{\!\{d_U(\pi_U(x),\pi_U(y))\}\!\}_{s_0+D+1} = \{\!\{d_S(\pi_S(x),\pi_S(y))\}\!\}_{s_0+D+1}.$$
Fix two points $t_1,t_2\in \mathcal{Y}$. If $d_S(\pi_S\circ\gamma(t_1),\pi_S\circ\gamma(t_2))\geq s_0+D+1$, then
$$ d_S(\pi_S\circ\gamma(t_1),\pi_S\circ\gamma(t_2))\asymp_{K,C} d_{\mathcal{X}}(\gamma(t_1),\gamma(t_2))\asymp_{\lambda,\varepsilon} d_{\mathcal{Y}}(t_1,t_2).$$
If $d_S(\pi_S\circ\gamma(t_1),\pi_S\circ\gamma(t_2))< s_0+D+1$, then clearly 
$$ d_S(\pi_S\circ\gamma(t_1),\pi_S\circ\gamma(t_2))\leq d_{\mathcal{Y}}(t_1,t_2)+s_0+D+1.$$
For the other quasi-geodesic inequality, first note there exists a uniqueness function $\theta$ associated to $(\mathcal{X},\mathfrak{S})$. If $d_{\mathcal{X}}(\gamma(t_1),\gamma(t_2))\geq \theta(s_0+D+1)$, then the uniqueness axiom implies there exists some domain $W\in\mathfrak{S}$ such that 
$$ d_W(\pi_W\circ\gamma(t_1),\pi_W\circ\gamma(t_2))\geq s_0+D+1.$$
However, this is a contradiction by the assumption that $d_S(\pi_S\circ\gamma(t_1),\pi_S\circ\gamma(t_2))< s_0+D+1$ and the fact that $\gamma$ has $D$-bounded projections. We then have that
$$ \frac{1}{\lambda} d_{\mathcal{Y}}(t_1,t_2)-\varepsilon-\theta(s_0+D+1)\leq d_{\mathcal{X}}(\gamma(t_1),\gamma(t_2))-\theta(s_0+D+1)\leq d_S(\pi_S\circ\gamma(t_1),\pi_S\circ\gamma(t_2)),$$
so by taking $k,c$ to be the maximum of the respective constants in the above inequalities, $\pi_S\circ\gamma$ is a $(k,c)$-quasi-isometric embedding for $k,c$ depending on $\lambda$, $\varepsilon$, and the hierarchy constants of $(\mathcal{X},\mathfrak{S})$.

For the reverse direction, let $\pi_S\circ\gamma$ be a $(k,c)$-quasi-isometric embedding into $CS$. We will show that $\gamma(\mathcal{Y})$ has $D$-bounded projections for the constant
$$ D = E(\lambda k(4M(k,c)+9E+4+c)+\varepsilon)+3E,$$
where $M$ is the Morse gauge of geodesics in an $E$-hyperbolic space. Fix any two points $x,y\in\gamma(\mathcal{Y})$ and any domain $U\in\mathfrak{S}-\{S\}$. If $d_U(\pi_U(x),\pi_U(y))\leq E$ then we are done, so suppose instead that $d_U(\pi_U(x),\pi_U(y))>E$. 

Let $\alpha$ be a $CS$-geodesic from $\pi_S(x)$ to $\pi_S(y)$. The bounded geodesic image axiom implies that $\alpha$ must intersect the $E$-neighborhood of $\rho^U_S$. Let $A\subseteq\alpha$ be the set of points in the $E$-neighborhood of $\rho^U_S$. Observe that $\text{diam}(A)\leq 3E$ because the diameter of $\rho^U_S$ is at most $E$. There are now three cases to consider, depending on the sizes of $d_S(\pi_S(x),A)$ and $d_S(\pi_S(y),A)$.

\textbf{Case 1:} Let $d_S(\pi_S(x),A)\leq M(k,c)+3E+2$ and $d_S(\pi_S(y),A)\leq M(k,c)+3E+2$. Then the triangle inequality implies
$$d_S(\pi_S(x),\pi_S(y))\leq d_S(\pi_S(x),A)+\text{diam}(A)+d_S(A,\pi_S(y)) \leq 2M(k,c)+9E+4.$$
However, $\pi_S\circ\gamma$ is a $(k,c)$-quasi-isometric embedding and $\gamma$ is a $(\lambda,\varepsilon)$-quasi-isometric embedding. Letting $s_1,s_2\in\mathcal{Y}$ be such that $\gamma(s_1)=x$ and $\gamma(s_2)=y$, it follows from $\pi_U$ being $(E,E)$-coarsely Lipschitz that
\begin{align*}
d_U(\pi_U(x),\pi_U(y))&\leq E\cdot d_{\mathcal{X}}(x,y)+E\\
&\leq E(\lambda\cdot d_{\mathcal{Y}}(s_1,s_2)+\varepsilon)+E\\
&\leq E(\lambda k(d_S(\pi_S(x),\pi_S(y))+c)+\varepsilon)+E\\
&\leq E(\lambda k(2M(k,c)+9E+4+c)+\varepsilon)+E\\
&\leq D,
\end{align*}
completing the first case.

\textbf{Case 2:} Suppose $d_S(\pi_S(x),A)>M(k,c)+3E+2$ and $d_S(\pi_S(y),A)>M(k,c)+3E+2$. Choose a point $z\in A$ such that $d_S(\pi_S(x),z)< d_S(\pi_S(x),A)+0.5$. Let $z'\in\alpha$ be the point such that
\begin{equation}
\label{final graph prod eq}
d_S(\pi_S(x),z')=d_S(\pi_S(x),z)-M(k,c)-3E-1,
\end{equation}
which exists because $d_S(\pi_S(x),A)>M(k,c)+3E+2$. Similarly, choose a point $w\in A$ such that $d_S(\pi_S(y),w)< d_S(\pi_S(y),A)+0.5$, and let $w'\in\alpha$ be the point such that $d_S(\pi_S(y),w')= d_S(\pi_S(y),w)-M(k,c)-3E-1$.

Because $\pi_S\circ\gamma$ is a $(k,c)$-quasi-isometric embedding into the $E$-hyperbolic space $CS$, Lemma \ref{qi hyp stable} implies that $\pi_S\circ\gamma$ is a $(M';k,c)$-stable embedding for some $M'$ depending on $k$, $c$, and $E$. From Lemma \ref{stable is morse} there exists and $(M'';k,c)$-Morse quasi-geodesic $\eta\subset\pi_S\circ\gamma(\mathcal{Y})$ with endpoints on $\alpha$, where $M''$ depends on $M'$, $k$, and $c$. Then because $\alpha$ is a geodesic in an $E$-hyperbolic space, there exist points $a,b\in\gamma(\mathcal{Y})$ such that $d_S(\pi_S(a),z')\leq M(k,c)$ and $d_S(\pi_S(b),w')\leq M(k,c)$. Consider a $CS$-geodesic from $\pi_S(a)$ to $\pi_S(x)$. Suppose towards contradiction that there exists a point $q$ along this geodesic which is contained in the $E$-neighborhood of $\rho^U_S$. Then $d_S(q,z)\leq 3E$. Using \eqref{final graph prod eq}, we have
\begin{align*}
d_S(\pi_S(x),z) &\leq d_S(\pi_S(x),q)+d_S(q,z)\\
&\leq d_S(\pi_S(x),\pi_S(a))+3E\\
&\leq d_S(\pi_S(x),z')+d_S(\pi_S(a),z')+3E\\
&\leq d_S(\pi_S(x),z')+M(k,c)+3E\\
&=d_S(\pi_S(x),z)-M(k,c)-3E-1+M(k,c)+3E\\
&=d_S(\pi_S(x),z)-1,
\end{align*}
which is a contradiction. Therefore, no $CS$-geodesic from $\pi_S(x)$ to $\pi_S(a)$ intersects the $E$-neighborhood of $\rho^U_S$, and so the bounded geodesic image axiom implies $d_U(\pi_U(x),\pi_U(a))<E$. An identical argument implies $d_U(\pi_U(y),\pi_U(b))<E$. By the triangle inequality,
\begin{align*}
d_S(\pi_S(a),\pi_S(b))&\leq d_S(\pi_S(a),z')+d_S(z',z)+\text{diam}(A)+d_S(w,w')+d_S(w',\pi_S(b))\\
&\leq M(k,c)+(M(k,c)+3E+1)+3E+(M(k,c)+3E+1)+M(k,c)\\
&= 4M(k,c)+9E+2.
\end{align*}
Let $t_1,t_2\in\mathcal{Y}$ be such that $\gamma(t_1)=a$ and $\gamma(t_2)=b$. Because $\pi_U$ is $(E,E)$-coarsely Lipschitz, the map $\gamma$ is a $(\lambda,\varepsilon)$-quasi-isometric embedding, and $\pi_S\circ\gamma$ is a $(k,c)$-quasi-isometric embedding, we have that
\begin{align*}
d_U(\pi_U(x),\pi_U(y))&\leq d_U(\pi_U(x),\pi_U(a))+d_U(\pi_U(a),\pi_U(b))+d_U(\pi_U(y),\pi_U(b))\\
&\leq d_U(\pi_U(a),\pi_U(b))+2E\\
&\leq E\cdot d_{\mathcal{X}}(a,b)+E+2E\\
&\leq E(\lambda\cdot d_{\mathcal{Y}}(t_1,t_2)+\varepsilon)+3E\\
&\leq E(\lambda k(d_S(\pi_S(a),\pi_S(b))+c)+\varepsilon)+3E\\
&\leq E(\lambda k(4M(k,c)+9E+2+c)+\varepsilon)+3E\\
&\leq D,
\end{align*}
completing the second case.

\textbf{Case 3:} Suppose only one of $d_S(\pi_S(x),A)$ or $d_S(\pi_S(y),A)$ is greater than $M(k,c)+3E+2$. Without loss of generality, let $d_S(\pi_S(x),A)>M(k,c)+3E+2$. Following the argument in case (2), construct the points $z,z'\in \alpha$ and $a\in\gamma(\mathcal{Y})$. By the triangle inequality,
\begin{align*}
d_S(\pi_S(a),\pi_S(y))&\leq d_S(\pi_S(a),z')+d_S(z',z)+\text{diam}(A)+d_S(A,\pi_S(y))\\
&\leq M(k,c)+(M(k,c)+3E+1)+3E+(M(k,c)+3E+2)\\
&=3M(k,c)+9E+3.
\end{align*}
Again let $t_1,t_2\in\mathcal{Y}$ be such that $\gamma(t_1)=a$ and $\gamma(t_2)=y$. As in Case 2, we have
\begin{align*}
d_U(\pi_U(x),\pi_U(y))&\leq d_U(\pi_U(x),\pi_U(a))+d_U(\pi_U(a),\pi_U(y))\\
&\leq E+d_U(\pi_U(a),\pi_U(y))\\
&\leq E+E\cdot d_{\mathcal{X}}(a,y)+E\\
&\leq E(\lambda\cdot d_{\mathcal{Y}}(t_1,t_2)+\varepsilon)+2E\\
&\leq E(\lambda k(d_S(\pi_S(a),\pi_S(y))+c)+\varepsilon)+2E\\
&\leq E(\lambda k(3M(k,c)+9E+3+c)+\varepsilon)+2E\\
&\leq D,
\end{align*}
completing the third case. Therefore, in any case, $\gamma$ has $D$-bounded projections.
\end{proof}

The following result brings together Theorem \ref{abd4.4} and Lemma \ref{bounded proj iff CS qgeo} to show that given the right initial conditions on a relative HHS $\mathcal{X}$, the top level space of the maximized relative HHS structure from Theorem \ref{clean containers ABD} is a Morse detectability space for $\mathcal{X}$, so $\mathcal{X}$ is Morse local-to-global.

\begin{thm}
\label{main thm}
Let $\mathcal{X}$ be a geodesic metric space and let $(\mathcal{X},\mathfrak{S})$ be a relative HHS with $|\mathfrak{S}|>1$, clean containers, the bounded domain dichotomy, and unbounded minimal products. A quasi-isometric embedding $\gamma:\mathcal{Y}\to\mathcal{X}$ is a stable embedding if and only if $\pi_S\circ\gamma$ is a quasi-isometric embedding into $\mathcal{T}_S$, where $\mathcal{T}_S$ is the top level space of the maximized structure.
\end{thm}

\begin{proof}
Let $\gamma:\mathcal{Y}\to\mathcal{X}$ be a quasi-isometric embedding. First, assume $\gamma$ is a stable embedding. By Theorem \ref{abd4.4}, $\gamma(\mathcal{Y})$ has $D'$-bounded projections. Thus Lemma \ref{bounded proj iff CS qgeo} implies $\pi_S\circ\gamma$ is a quasi-isometric embedding.

For the opposite direction, suppose $\pi_S\circ\gamma$ is a quasi-isometric embedding. Lemma \ref{bounded proj iff CS qgeo} implies that $\gamma(\mathcal{Y})$ has $D$-bounded projections. Theorem \ref{abd4.4} then implies that $\gamma(\mathcal{Y})$ is $D'$-contracting for some $D'$. Because $\gamma(\mathcal{Y})$ is contracting, it is a stable embedding by \cite[Corollary 4.3]{durham2015convex}. 
\end{proof}

\begin{cor}
\label{abd6.2}
Let $\mathcal{X}$ be a geodesic metric space and let $(\mathcal{X},\mathfrak{S})$ be a relative HHS with $|\mathfrak{S}|>1$, clean containers, the bounded domain dichotomy, and unbounded minimal products. Then $\mathcal{X}$ is Morse local-to-global.    
\end{cor}

\begin{proof} The $E$-relative HHS $(\mathcal{X},\mathfrak{S})$ has the bounded domain dichotomy and unbounded minimal products, so Theorem \ref{clean containers ABD} implies there exists an $E'$-relative HHS structure $(\mathcal{X},\mathfrak{T})$ with unbounded products. Moreover, the construction of $\mathfrak{T}$ yields an $(E',E')$-coarsely Lipschitz map $\pi_S\colon\mathcal{X}\to\mathcal{T}_S$, where $\mathcal{T}_S$ is $E'$-hyperbolic. We will show that $\mathcal{X}$ is Morse detectable. 

For the first condition of Definition \ref{morse detectable def}, fix a $(M;\lambda,\varepsilon)$-Morse quasi geodesic $\gamma\colon I\to\mathcal{X}$. Then $\gamma$ is a stable embedding by Lemma \ref{morse is stable}. Thus Theorem \ref{main thm} implies $\pi_S\circ\gamma$ is a $(k,c)$-quasi-geodesic, where $k$ and $c$ are determined by $\lambda$, $\varepsilon$, $M$, and the hierarchy constants of $(\mathcal{X},\mathfrak{S})$.

For the second condition, let $\gamma\colon I\to\mathcal{X}$ be a $(\lambda,\varepsilon)$-quasi-geodesic such that $\pi_S\circ\gamma$ is a $(k,c)$-quasi-geodesic in $CS$. Theorem \ref{main thm} implies $\gamma$ is an $M'$-stable embedding, so in particular it is $M'$-Morse for some Morse gauge $M'$. Condition (2) of Definition \ref{morse detectable def} is thus satisfied. Therefore we have shown that $\mathcal{X}$ is Morse detectable, and hence it is Morse local-to-global by Theorem \ref{Mdetectable is MLTG}.
\end{proof}

\section{The Morse Local-to-Global Property for Graph Products}
\label{sec MLTG}

This section will ultimately show that graph products of infinite Morse local-to-global groups are Morse local-to-global. To begin, we specified in Theorem \ref{clean containers ABD} that the initial relative HHS structure should have clean containers both because the resulting structure is a genuine relative HHS, and because graph products, when viewed as an HHG, admit clean containers, as seen in the following proposition.

\begin{prop}
\label{graph prod clean}
Graph products admit a relative HHG structure with clean containers.
\end{prop}

\begin{proof} Let $G_{\Gamma}$ be a graph product. Equip $G_{\Gamma}$ with the relative HHG structure from Theorem \ref{berlyne russell}. Consider parallelism classes $[h\Omega]\sqsubsetneq [g\Lambda]$ and $[k\Pi]$ such that $[k\Pi]\sqsubseteq [g\Lambda]$ and $[k\Pi]\perp[h\Omega]$. Then \cite[Lemma 4.6]{berlyne2022hierarchical} yields the container $[a(\text{lk}(\Omega)\cap\Lambda)]$ where $a\in G_{\Gamma}$ satisfies $[a\Lambda]=[g\Lambda]$ and $[a\Omega]=[h\Omega]$. Clearly $\text{lk}(\Omega)\cap\Lambda\subset\text{lk}(\Omega)$, and $[a\Omega]=[h\Omega]$ by construction, so $[a(\text{lk}(\Omega)\cap\Lambda)]\perp[h\Omega]$ by \cite[Theorem 3.23]{berlyne2022hierarchical}. Therefore $G_{\Gamma}$ has clean containers.
\end{proof}

\begin{prop}
\label{graph prod structure}
Graph products of infinite groups with no isolated vertices admit a relative HHG structure with $|\mathfrak{S}|>1$, clean containers, the bounded domain dichotomy, and unbounded minimal products.
\end{prop}

\begin{proof} Let $G$ be a graph product with an associated finite simplicial graph $\Gamma$ with no isolated vertices and let all the vertex groups of $G$ be infinite. By \cite[Theorem 4.22]{berlyne2022hierarchical} there is a relative HHG structure on $G$ for which domains are, in the language of \cite{berlyne2022hierarchical}, parallelism classes of cosets $g\Lambda$, where $\Lambda$ is a subgroup corresponding to the subgraph $\Lambda\subseteq\Gamma$. In particular, because $\Gamma$ has no isolated vertices, it has more than one vertex, so $|\mathfrak{S}|>1$. Moreover, $G$  has the bounded domain dichotomy by the definition of a relative HHG and clean containers by Proposition \ref{graph prod clean}. Finally, let $[g\Lambda]\in\mathfrak{S}$ be a $\sqsubseteq$-minimal domain with $\mathbf{F}_{[g\Lambda]}$ unbounded. The domain $[g\Lambda]$ is $\sqsubseteq$-minimal, which means that $\Lambda$ contains no proper subgraphs, so it is a single vertex. Moreover, $\Gamma$ is has no isolated vertices, so there exists another vertex connected to $\Lambda$, whose subgraph we will label by $\Omega$. Then by \cite[Theorem 3.23]{berlyne2022hierarchical}, the domains $[g\Lambda]$ and $[g\Omega]$ are orthogonal. Because each vertex corresponds to an infinite group by assumption, $C([g\Omega])$ is unbounded, so $\mathbf{E}_{[g\Lambda]}$ is unbounded. Therefore $G$ has unbounded minimal products. Thus the Cayley graph of $G$ is a geodesic metric space and a relative HHS with $|\mathfrak{S}|>1$, clean containers, the bounded domain dichotomy, and unbounded minimal products.
\end{proof}

The following corollary follows immediately from Proposition \ref{graph prod structure} and Corollary \ref{abd6.2}.

\begin{cor}
\label{multivertex}
Graph products of infinite groups with no isolated vertices are Morse local-to-global.
\end{cor}

For a graph product with no isolated vertices, the top level space in the maximized structure can be described explicitly. This space, as described in the following corollary, was known to be hyperbolic by \cite[Proposition 6.4]{genevois2024automorphisms} and \cite[Lemma 4.1]{berlyne2022hierarchical}. Corollary \ref{bckmt} also mirrors \cite[Theorem 1.2]{balasubramanya2025stable}, which was proven simultaneously and independently, and provides a similar space for a graph product of infinite groups with no isolated vertices. 

\begin{cor}
\label{bckmt}
Let $G_\Gamma$ be a graph product of infinite groups with no isolated vertices. Let $S$ be a finite generating set for $G_{\Gamma}$, and let $H\leq G_{\Gamma}$ be a finitely generated subgroup. Then $H$ is stable if and only if the orbit maps of $H$ into
$$ \text{\normalfont{Cay}}\left(G_{\Gamma}, \bigcup_{\{\Lambda\subset\Gamma\;|\;\text{\normalfont{lk}}(\Lambda)\ne\emptyset\}}G_{\text{\normalfont{st}}(\Lambda)}-\{\text{\normalfont{id}}\}\right)$$
are quasi-isometric embeddings.
\end{cor}

\begin{proof}
Let $(G_{\Gamma},\mathfrak{S})$ be the relative HHG structure on $G_{\gamma}$ and let $(G_{\Gamma},\mathfrak{T})$ be the relative HHG structure with unbounded products from Proposition \ref{graph prod structure} and Theorem \ref{clean containers ABD}. Let $\mathcal{T}_S$ be the top level space for $(G_{\Gamma},\mathfrak{T})$, so Theorem \ref{main thm} implies subgroup $H\leq G_{\Gamma}$ is stable if and only if the orbit maps of $H$ into $\mathcal{T}_S$ are quasi-isometric embeddings. All that remains to show is that $\mathcal{T}_S$ is obtained by coning-off the subgroups associated to induced subgraphs with nonempty links. Following the construction in Theorem \ref{clean containers ABD}, the top level space $\mathcal{T}_S$ is obtained by coning off the slices $\mathbf{F}_U\times\{\Vec{e}\}$ for all $U$ with both $\mathbf{F}_U$ and $\mathbf{E}_U$ unbounded, so the product regions of these domains will be coned-off. For a multi-vertex connected graph product of infinite groups, every parallelism class has unbounded $\mathbf{F}_U$ because $\mathbf{F}_U$ contains the Cayley graph of at least one vertex group, all of which are infinite. For some parallelism class $U=[g\Lambda]$ to have unbounded $\mathbf{E}_U$, there must exist another parallelism class $V=[h\Omega]$ such that $[g\Lambda]\perp [h\Omega]$ and $\mathfrak{T}_V$ contains a domain with an infinite associated geodesic space, which again will always happen as long as $[h\Omega]$ exists because every vertex group is infinite. Recall that Definition \ref{graph prod relations} states $[g\Lambda]\perp [h\Omega]$ if and only if $\Lambda\subseteq\text{lk}(\Omega)$ and and there exists a group element $k\in G_{\Gamma}$ such that $[g\Lambda]=[k\Lambda]$ and $[h\Omega]=[k\Omega]$. Thus the parallelism classes with unbounded $\mathbf{E}_U$ are exactly the those with a non-empty link. Because the domains associated to the product region of a parallelism class are exactly those in the star of its induced subgraph, the resultant space is exactly that in the statement of the corollary.
\end{proof}

It is worth noting that the spaces in \cite[Theorem 1.2]{balasubramanya2025stable} and this paper are quasi-isometric. The space in \cite[Theorem 1.2]{balasubramanya2025stable} is the contact graph of the prism complex, which is quasi-isometric to the space given by coning off cosets of the star graphs of vertices \cite{genevois2017cubical, genevois2019automorphisms}, in the language of \cite[Lemma 2.8]{balasubramanya2025stable}. We instead cone off cosets of the star graphs of all induced subgraphs with non-empty links. 

\begin{prop} Let $G_\Gamma$ be a connected graph product. Then
$$ \text{\normalfont{Cay}}\left(G_{\Gamma}, \bigcup_{\{\Lambda\subset\Gamma\;|\;\text{\normalfont{lk}}(\Lambda)\ne\emptyset\}}G_{\text{\normalfont{st}}(\Lambda)}-\{\text{\normalfont{id}}\}\right)
\simeq_{\text{\normalfont{Q.I.}}} 
\text{\normalfont{Cay}}\left(G_{\Gamma}, \bigcup_{v\in V(\Gamma)}G_{\text{\normalfont{st}}(v)}-\{\text{\normalfont{id}}\}\right).$$
\end{prop}

\begin{proof} Label the cone-offs
$$ A = \text{\normalfont{Cay}}\left(G_{\Gamma}, \bigcup_{\{\Lambda\subset\Gamma\;|\;\text{\normalfont{lk}}(\Lambda)\ne\emptyset\}}G_{\text{\normalfont{st}}(\Lambda)}-\{\text{\normalfont{id}}\}\right);B = \text{\normalfont{Cay}}\left(G_{\Gamma}, \bigcup_{v\in V(\Gamma)}G_{\text{\normalfont{st}}(v)}-\{\text{\normalfont{id}}\}\right).$$
Because $G_{\Gamma}$ is a connected graph product, a single vertex is a subgraph with a non-empty link, so $A$ is a cone-off of $B$. It remains to show that for any $\Lambda\subset\Gamma$ with non-empty link, the length of any word contained entirely in $G_{\text{st}(\Lambda)}$ has length bounded by a uniform constant in $B$. Let $g=g_1g_2\cdots g_n$ be some word in $G_{\Gamma}$ contained entirely in $G_{\text{st}(\Lambda)}=G_{\Lambda}\times G_{\text{lk}(\Lambda)}$. Each element of $g$ is either contained in $G_{\Lambda}$ or $G_{\text{lk}(\Lambda)}$, and $g$ can be written as $g = h_1k_1h_2k_2\cdots h_mk_m$, where $h_i\in G_{\Lambda}$ and $k_i\in G_{\text{lk}(\Lambda)}$ for $i\in\{1,2,...,m\}$. Since elements of $G_{\Lambda}$ and $G_{\text{lk}(\Lambda)}$ commute, $g$ can be rearraged to be $h_1h_2\cdots h_mk_1k_2\cdots k_m$. Fix vertices $v,w$ such that $v\in \text{lk}(\Lambda)$ and $w\in \Lambda$. Then $h_1h_2\cdots h_m\in \text{st}(v)$ and $k_1k_2\cdots k_m\in\text{st}(w)$. Thus $g$ has length 2 in $B$, so $A$ and $B$ are quasi-isometric.    
\end{proof}

Finally, we generalize Corollary \ref{multivertex} to the case of graph products that have isolated vertices, when every vertex group is an infinite Morse local-to-global group.

\begin{cor}
Graph products of infinite Morse local-to-global groups are Morse local-to-global.
\end{cor}

\begin{proof} Suppose $\Gamma$ is the finite simplicial graph associated to the graph product $G$. By the definition of a graph product, $G$ is the free product of its connected components. Thus if each subgroup associated to a connected component of $\Gamma$ is Morse local-to-global, then $G$ is Morse local-to-global by \cite[Theorem 5.1]{russell2022local}. Let $H$ be the subgroup associated to a single connected component $\Lambda\subseteq\Gamma$. If $\Lambda$ is a single vertex, then $H$ is Morse local-to-global by assumption. If $\Lambda$ contains more than one vertex, then $H$ is a graph product of infinite groups with no isolated vertices, and is Morse local-to-global by Corollary \ref{multivertex}.    
\end{proof}

\bibliographystyle{alpha}
\bibliography{M335}

\end{document}